\documentclass[12pt]{article}
\usepackage[usenames]{color}
\usepackage{amsmath,amssymb,amsthm}
\usepackage[all]{xy}
\usepackage[normalem]{ulem}
\xyoption{arc}

\numberwithin{equation}{section}

\newtheorem{thm}{Theorem}[section]
\newtheorem{cor}[thm]{Corollary}
\newtheorem{prop}[thm]{Proposition}
\newtheorem{lem}[thm]{Lemma}
\theoremstyle{definition}
\newtheorem{defn}[thm]{Definition}
\theoremstyle{remark}
\newtheorem{rmk}[thm]{Remark}

\def\co{\colon\thinspace}
\newcommand{\mb}[1]{\mathbb{#1}}
\newcommand{\mf}[1]{\mathfrak{#1}}

\newcommand{\Ext}{\ensuremath{{\rm Ext}}}

\newcommand{\Hom}{\ensuremath{{\rm Hom}}}

\newcommand{\holim}{\ensuremath{\mathop{\rm holim}}}
\newcommand{\coker}{\ensuremath{\mathop{\rm coker}}}
\newcommand{\overto}{\mathop\rightarrow}

\newcommand{\Map}{\ensuremath{{\rm Map}}}

\newcommand{\TAF}{{\rm TAF}}

\newcommand{\hAut}{{\rm hAut}}

\newcommand{\Spec}{{\rm Spec}}

\newcommand{\Spf}{{\rm Spf}}
\newcommand{\tmf}{{\rm tmf}}

\newcommand{\TMF}{{\rm TMF}}

\newcommand{\Tr}{{\rm Tr}}

\newcommand{\psth}{\psi\mbox{-}\theta}
\newcommand{\alg}{\mbox{-}alg}

\newcommand{\BPP}[1]{{\rm BP}{\left\langle#1\right\rangle}}

\newcommand{\eilm}[1]{\ensuremath{{\mb H} #1}}
\newcommand{\smsh}[1]{\ensuremath{\mathop{\wedge}_{#1}}}

\newcommand{\comp}[1]{\ensuremath{#1^\wedge}}

\newcommand{\pow}[1]{[\![{#1}]\!]}
\newcommand{\laur}[1]{\left(\!\left({#1}\right)\!\right)}

\title{
Commutativity conditions for truncated Brown-Peterson
spectra of height $2$}
\author{Tyler Lawson\thanks{Partially supported by NSF
    grant 0805833 and a fellowship from the Sloan foundation.}, Niko Naumann}

\begin{document}
\maketitle
\begin{abstract}
An algebraic criterion, in terms of closure under power operations, is
determined for the existence and uniqueness of generalized
truncated Brown-Peterson spectra of height $2$ as $E_\infty$-ring
spectra.  The criterion is checked for an example at the prime
$2$ derived from the universal elliptic curve equipped with a level
$\Gamma_1(3)$ structure.

{\sf 2000MSC: 55P42, 55P43, 55N22 and 14L05}
\end{abstract}

\tableofcontents

\section{Introduction}

The truncated Brown-Peterson spectra $\BPP{n}$ associated to a prime
$p$ and an integer $n$, introduced in
\cite{johnson-wilson-projective}, are connective versions of the
Landweber exact spectra $E(n)$ developed by Johnson and Wilson.  The
coefficient ring $\pi_*\BPP{n}=\mb Z_{(p)}[v_1,\ldots,v_n]$
parameterizes a simple family of formal group laws concentrated at
chromatic heights $0$ through $n$ and at height $\infty$.

These spectra $\BPP{n}$ have enjoyed recent prominence for the role
they play in algebraic $K$-theory.  Ausoni and Rognes
\cite[Introduction (0.2)]{ausoni-rognes} expect the existence of
algebraic $K$-theory localization sequences,
\begin{equation}\label{eq:Rognes}
K(\comp{\BPP{n-1}}_p) \to K(\comp{\BPP{n}}_p) \to K(\comp{E(n)}_p),
\end{equation}
as part of a program to understand algebraic $K$-theory by
applying it to the chromatic tower.  For this and related
applications it is useful to know whether these spectra admit highly
associative ($A_\infty$-) or highly commutative ($E_\infty$-) ring
structures.  They are all known to admit $A_\infty$-ring structures
\cite[discussion following Remark 2.12]{lazarev-towers}.

Only the simplest examples are known to admit $E_\infty$-ring
structures: the spectra $\BPP{-1} = \eilm{\mb Z/p}$, $\BPP{0} =
\eilm{\mb Z_{(p)}}$, and the spectra $\BPP{1}$
\cite{mcclure-staffeldt}, which are the Adams summands of connective
$K$-theory. We remark that in exactly these cases, the existence of a
localization sequence (\ref{eq:Rognes}) is known.  For $n=0$ this is
classical, due to Quillen, and for $n=1$ it is work of Blumberg
  and Mandell \cite[Introduction]{blumbergmandell}.

In this paper we construct a generalized $\BPP{2}$ as an
$E_\infty$-ring spectrum at the prime $2$, as follows.

\begin{thm}
\label{thm:main}
There exists a $2$-local complex oriented $E_\infty$-ring spectrum
$\tmf_1(3)_{(2)}$ such that the composite map of graded rings
\[
\mb Z_{(2)}[v_1, v_2] \subseteq BP_* \to (MU_{(2)})_* \to
\pi_*(\tmf_1(3)_{(2)}) 
\]
is an isomorphism. Here the $v_i$ denote the $2$-primary Hazewinkel
generators.  Any other such spectrum with an isomorphic formal group
law is equivalent to $\tmf_1(3)_{(2)}$ as an $E_\infty$-ring spectrum.
\end{thm}
In Section~\ref{sec:realization} we will discuss realization
problems, and in particular how the existence of the above isomorphism
is intrinsic to the formal group.

In forthcoming work \cite{tmforientation}, the authors will show that
there is a commutative diagram of $E_\infty$-ring spectra
\[
\xymatrix{
\tmf_{(2)} \ar[r]\ar[d] & ko_{(2)}\ar[d] \\
\tmf_1(3)_{(2)} \ar[r]^(.6){\alpha} & ku_{(2)}
}
\]
which, in mod $2$ cohomology, induces the following canonical diagram
of modules over the mod $2$ Steenrod algebra ${\cal A}$:
\[
\xymatrix{
{\cal A}/\!/{\cal A}(2) & 
{\cal A}/\!/{\cal A}(1)\ar[l]\\
{\cal A}/\!/E(2)\ar[u]& 
{\cal A}/\!/E(1)\ar[l]\ar[u]
}
\]
On homotopy groups, $\alpha$ induces the map sending $v_1$ to $v_1$
and $v_2$ to zero, and there is a consequent cofiber sequence of
$\tmf_1(3)_{(2)}$-modules
\[
\Sigma^6 \tmf_1(3)_{(2)}\stackrel{\cdot v_2}{\longrightarrow} \tmf_1(3)_{(2)}
\stackrel{\alpha}{\longrightarrow} ku_{(2)}.
\]
The authors are grateful to Rognes for pointing out that this result
might be useful in on-going joint work of Bruner and Rognes aimed at
computing the algebraic $K$-groups $\pi_* K(\tmf_{(2)})$.

In general, the possibility of constructing a truncated Brown-Peterson
spectrum $\BPP{n}$ as an $E_\infty$-ring spectrum can be determined by
iterative application of obstruction theory through a chromatic
fracture cube.  The necessary methods to carry this out are already
well-understood, and were significantly applied in the construction of
the spectrum of topological modular forms by Goerss, Hopkins, and
Miller \cite{goerss-hopkins,behrens-construct}.  These methods, when
applied to $\BPP{2}$, break down into the following stages:
\begin{itemize}
\item The Goerss-Hopkins-Miller theorem produces the
  $K(2)$-localized $E_\infty$-ring spectrum $L_{K(2)} \BPP{2}$ as the
  homotopy fixed point object of a finite group action on a
  Lubin-Tate spectrum.
\item The homotopy groups of the $K(1)$-localization $L_{K(1)}
  L_{K(2)} \BPP{2}$ are equipped with a power operation $\theta$.
  These homotopy groups contain a subring consisting of the homotopy
  groups of $L_{K(1)} \BPP{2}$.  If this is a $\theta$-invariant
  subring, one produces an $E_\infty$-ring spectrum $L_{K(1)} \BPP{2}$
  equipped with a chromatic attaching map by use of Goerss-Hopkins
  obstruction theory \cite{goerss-hopkins}.  (If this is {\em not} a
  $\theta$-invariant subring, then $\BPP{2}$ does not admit an
  $E_\infty$-ring structure.)
\item The $E_\infty$-ring spectrum $L_{K(1) \vee K(2)} \BPP{2}$ is
  then constructed by a chromatic pullback square.
\item The homotopy groups of $L_{K(0)} L_{K(1) \vee K(2)} \BPP{2}$
  contain as a subring the homotopy groups of $L_{K(0)} \BPP{2}$.
  The latter is a free algebra on two generators, and this allows $L_{K(0)}
  \BPP{2}$ to be realized as an $E_\infty$-ring spectrum, along with
  an arithmetic attaching map, by methods of rational homotopy theory.
\item The object $L_{K(0) \vee K(1) \vee K(2)} \BPP{2}$ is then
  constructed by a chromatic pullback square.
\item $\BPP{2}$ is the connective cover of $L_{K(0) \vee K(1) \vee
    K(2)} \BPP{2}$, and canonically inherits an $E_\infty$-ring
  structure from $L_{K(0) \vee K(1) \vee K(2)} \BPP{2}$ by \cite[VII, 4.3]{ekmm}.
\end{itemize}

More generally, these steps can be followed without starting with
$\BPP{2}$.  Instead, one starts with a graded ring carrying a formal
group law with similar chromatic data (but possibly determined as a
quotient by a different regular sequence in $MU_*$).  We will
formalize this with the notion of a generalized $\BPP{n}$- realization
problem in Definition \ref{defn:realization}.  If this realization
  problem can be solved, the solution is a {\em generalized}
truncated Brown-Peterson spectrum as considered previously, such as in
\cite{strickland-products, baker-johnsonwilson}.

The main result of this paper is Theorem~\ref{thm:newmain}.  It
shows that a formal group law over the graded ring $\mb
Z_{(p)}[v_1,v_2]$, chromatically similar to that for $\BPP{2}$, can be
realized by an $E_\infty$-ring spectrum $R$ (which is essentially
unique) if and only if the ring $\comp{\mb Z[(v_2/v_1^{p+1})]}_p$ is
stable under an induced power operation $\theta$.  Moreover, any
  solution is unique up to weak equivalence. (This ring is a
subring of the homotopy of a corresponding $K(1)$-localized Lubin-Tate
spectrum). Work initiated by Ando \cite{ando-isogenies,
ando-hopkins-strickland} allows an algebro-geometric interpretation
of $\theta$ in terms of the associated formal group laws and descent
data for level structures. It is perhaps surprising that the only
problematic part of this general realization program is purely
algebraic.

The proof of Theorem \ref{thm:newmain} occupies the bulk of the paper.
By work of Rezk \cite{rezk-height2}, at the prime $2$ the power
operations on the ring of modular forms of level $\Gamma_1(3)$ turn
out to provide precisely the necessary algebraic descent data for
level structures, and Theorem~\ref{thm:main} follows.  There are
currently few instances where explicit computations of power
operations on Lubin-Tate spectra of chromatic height $2$ are known,
and this is one of the main obstructions to extending Theorem
\ref{thm:main} to primes $p\neq 2$ or to other generalized
$\BPP{2}$.

We remark that methods involving modular curves cannot extend to
construct generalized versions of $\BPP{2}$ at arbitrary primes.
There are only finitely many modular curves having an underlying
coarse moduli space which is a curve of genus zero with exactly one
supersingular point, as demanded by the graded ring $\BPP{2}_*$
together with its formal group.  For the prime $p=3$ a similar method
based on (derived) compact Shimura curves is shown to work in
\cite[Theorem 4.2]{hill-lawson}.  On the other hand, Rezk's explicit
formula in Proposition~\ref{prop:theta} is strictly more data than
necessary.  The final result does not require elliptic curves at all,
but only a height $2$ formal group law $\mb G$ over $\mb Z_p\pow a$
and a small subring of $\comp{\mb Z\laur a}_p$ invariant under a lift
of Frobenius which is canonically attached to $\mb G$.  Using
crystalline Diedonn\'e theory, this existence statement translates
into one concerning a rank $2$ $F$-crystal over $\mb Z_p\pow{a}$
with specific rationality properties.  Though rather explicit, this
problem remains open at present for all primes $p\neq 2,3$.

The current work is connected with the multiplicative ring spectrum
project of Goerss and Hopkins \cite{goerssnewton} which attempts to
lift algebraic diagrams to diagrams of $E_\infty$-ring spectra. The
strongest currently available results in this direction are due to
Lurie \cite[Theorem 4.7]{Pauloverview}, and roughly produce the
required $E_\infty$-ring spectra from suitable $p$-divisible groups.
These in turn can be produced from families of elliptic curves
(leading to $\TMF$ \cite{lurie-survey}) and, more generally, from
suitable abelian varieties (leading to $\TAF$ \cite{taf}).

However, there are many more $p$-divisible groups than there are
abelian varieties.  This flexibility allowed specifically constructed
$p$-divisible groups to realize the localizations
$L_{K(2)\vee\cdots\vee K(n)} E(n)$ of generalized Johnson-Wilson
spectra as $E_\infty$-ring spectra \cite[Example 5.4]{tylerdisplays}
Given appropriate algebraic data, the spectrum
$L_{K(1)} E(n)$ and its chromatic attaching map are directly
realizable by the Goerss-Hopkins obstruction theory, and the question
of whether an $E_\infty$-ring structure can be placed on the spectrum
$E(n)$ is reduced to a very explicit, though potentially hard,
arithmetic existence problem.

We conclude this section by giving an overview of the paper.

In Section \ref{sec:power-operations} we recall fundamental results
about power operations with an emphasis on the $K(1)$-local case.

Section \ref{sec:realization} sets up the realization problems for
generalized $\BPP{n}$ and introduces some technical material: for
understanding how complex orientations can be lifted along a map of
ring spectra, and for rigidifying certain group actions.

Section \ref{sec:k2-localization} shows that, as a consequence of the
Goerss-Hopkins-Miller theorem and realizability results of Baker and
Richter, the realization problem for $\BPP{2}$ always has an
essentially unique solution $K(2)$-locally, given by a homotopy fixed
point spectrum of a Lubin-Tate spectrum.

Section \ref{sec:k1-local-obstruction}, studying
$K(1)$-local obstruction theory, is the technical heart of
the paper.  It leads up to the proof of the main $K(1)$-local
existence and uniqueness result, Theorem \ref{thm:realizability}.

Section \ref{sec:k1-localization} deduces that closure under the
power operation $\theta$ is necessary and sufficient to extend a
$K(2)$-local solution to a given existence problem to a $K(1) \vee
K(2)$-local solution, and that any such extension is unique up to
equivalence.

Section \ref{sec:k0-localization} completes the proof characterizing
existence and proving uniqueness for realizing a generalized
$\BPP{2}$, Theorem~\ref{thm:newmain}, by methods of rational homotopy
theory.

Section \ref{sec:proofofmain} provides a purely algebraic description
of the power operation $\theta$ and uses elliptic curves with
$\Gamma_1(3)$-level structures to complete the proof of Theorem
\ref{thm:main}.

Throughout this paper, we follow the standard conventions that, for a
spectrum $E$, we write $\pi_* E = E_* = E^{-*}$.

\subsection*{Acknowledgements}

The authors would like to thank Andrew Baker, Daniel Davis, and John
Rognes for discussions of some of the present material, and the
anonymous referees for a number of observant remarks and the
suggestion to study uniqueness.

\section{Power operations}
\label{sec:power-operations}

Let $p$ be a prime and $E$ an $E_\infty$-ring spectrum which is
$p$-local and complex orientable.  In this section we review the
necessary details from \cite{ando-hopkins-strickland} about power
operations on $E$-cohomology. The authors have been heavily influenced
by the work of Rezk \cite{rezk-power-operations}.

\subsection{Total power operations}
For any element $\alpha \in E^0(X) = [\Sigma^\infty_+ X, E]$, composing
the $p$-fold symmetric power of this map with the multiplication map
of $E$ gives rise to a power operation
\[
P(\alpha)\co \Sigma^\infty_+ (B\Sigma_p \times X) \to E.
\]
The transformation $P\co E^0(X) \to E^0(X \times B\Sigma_p)$ is
natural in spaces $X$ and $E_\infty$-ring spectra $E$.  It preserves
multiplication but not addition (there is a Cartan formula),
and composing $P(\alpha)$ with the inclusion $X \to B \Sigma_p \times
X$ gives $\alpha^p$.

Let $I_X$ be the ideal which is the image of the transfer $\Tr\co E^0(X)
\to E^0(X \times B\Sigma_p)$; then $I_X$ is the principal ideal
generated by $\Tr(1) \in E^0(B \Sigma_p)$.  As $E$ is $p$-local, this
ideal coincides with the image of the transfer from $E^0(X \times
B(\Sigma_r \times \Sigma_s))$ associated to any subgroup $\Sigma_r
\times \Sigma_s$ for $0 < r,s < p$, $r+s = p$.  The resulting natural
transformation
\[
\xymatrix{ \psi^p_X\co E^0(X) \ar[r]^P & E^0(X\times {B\Sigma_p})\ar@{>>}[r] & E^0(X \times {B\Sigma_p})/I_X}\\
\]
is a ring homomorphism.  We simply write $\psi^p = \psi^p_X$ when $X$
is a point.

The image of $\Tr(1)$ under the map $E^0(B\Sigma_p) \to E^0(*)$ is
$p!$, and hence we obtain a factorization as follows:
\[
\xymatrix{
E^0(X) \ar[r]^{\psi^p\hskip 2.5pc} \ar[dr]_{(-)^p} &
E^0(X \times B\Sigma_p)/I_X \ar[d] \\
& E^0(X)/(p)
}
\]

The ring homomorphism $\psi^p$ therefore reduces, modulo the
ideal $(p)$, to the Frobenius ring homomorphism $x \mapsto x^p$
on $E^0(X)/(p)$.

Let $j\co C_p \hookrightarrow \Sigma_p$ be a monomorphism from a
cyclic subgroup of order $p$.  The index $[\Sigma_p:C_p]$ is prime to
$p$, so the image of $I_X$ under the restriction map $E^0(X \times
B\Sigma_p) \to E^0(X \times BC_p)$ coincides with the image of the
transfer morphism $E^0(X) \to E^0(X \times BC_p)$.

Consider the case $X = BC_p$.  Write $\epsilon\co E^0(BC_p) \to E^0$
for the map induced by the homomorphism $\{e\} \to C_p$ from the
trivial group.  For any $\alpha \in E^0(BC_p)$, we have the
following diagram, which commutes up to homotopy:
\[
\xymatrix{
\Sigma^\infty_+BC_p \ar[d]_{(1 \times j) \circ \Delta} \ar[r]^j &
\Sigma^\infty_+B(\{e\} \wr \Sigma_p) \ar[d] \ar[r]^{\hskip 1pc{\epsilon(\alpha)}^p} &
E^{\smsh{} p}_{h\Sigma_p} \ar@{=}[d] \ar[r] &
E \ar@{=}[d]
\\
\Sigma^\infty_+(BC_p \times B\Sigma_p) \ar[r]_{\Delta} &
\Sigma^\infty_+B(C_p \wr \Sigma_p) \ar[r]_{\hskip 1pc \alpha^p} &
E^{\smsh{} p}_{h\Sigma_p} \ar[r] &
E
}
\]
(Here the wreath product $G \wr \Sigma_p$ denotes the semidirect
product $G^p \rtimes \Sigma_p$, with the symmetric group acting by
permutation.)
The left-hand square commutes up to homotopy because the two composite
maps are induced by two conjugate group homomorphisms $C_p \to C_p \wr
\Sigma_p$.  As a consequence, we have the following commutative
diagram of total power operations:
\begin{equation}
  \label{eq:totalpowers}
\xymatrix{
E^0(BC_p) \ar[r]^{P{\hskip 2pc}} \ar[d]_{\epsilon} &
E^0(BC_p \times B\Sigma_p) \ar[d]^{\Delta^* \circ (1 \times j)^*} \\
E^0 \ar[r]_{j^* \circ P} &
E^0(BC_p)
}
\end{equation}

\subsection{Even-periodic objects}

In this section we assume that our fixed $p$-local, complex orientable
$E_\infty$-ring spectrum $E$ is even-periodic, i.e. that for all
integers $m$, $E^{2m+1}$ is zero and there is a unit in $E^2$.

Under these circumstances, the function spectrum $E^X =
F(\Sigma^\infty_+X,E)$ is an $E_\infty$-ring spectrum which inherits a
complex orientation from $E$.  The ring $E^0(X \times
\mb{CP}^\infty)$ is the coordinate ring of a formal group $\mb G$ over
$E^0(X)$ \cite[Section 8]{stricklandformal}.  Naturality of the map
$\psi^p_{(-)}$ with respect to the coordinate projections and
multiplication on $\mb{CP}^\infty \times \mb{CP}^\infty$ implies that
the map
\[
\psi^p_{\mb{CP}^\infty}\co E^0(\mb{CP}^\infty) \to E^0(\mb{CP}^\infty
\times B\Sigma_p)/I_{\mb{CP}^\infty}
\]
is a homomorphism $\mb G \to (\psi^p)^* \mb G$ of formal groups which
is a lift of the Frobenius isogeny.  More is true.  When the formal
group has constant height, the ring $E^0(B\Sigma_p)/I$ is the
universal $E^0$-algebra classifying subgroup schemes of $\mb G$ of
rank $p$, and consideration of equation~(\ref{eq:totalpowers})
implies that this universal subgroup scheme is the kernel of the
isogeny $\psi^p_{\mb{CP}^\infty}$.

Let $C_p \subseteq S^1$ be the unique cyclic subgroup of order $p$.
It induces a map $BC_p\to BS^1\cong \mb{CP}^\infty$.  If $p$ is not a
zero divisor in $E^*$, the map $E^0(\mb{CP}^\infty) \to E^0(BC_p)$ is
a quotient map, and $E^0(BC_p)$ is the coordinate ring of the
$p$-torsion subgroup of the formal group.

\subsection{Homogeneous objects}
\label{sec:homogeneous}

In this section we assume that $E$ is even-periodic, that the ring
$E^0$ is a $p$-torsion-free complete local ring with maximal ideal
$\mf m$ containing $p$, and that the mod-$\mf m$ reduction of the
formal group on $E^0$ is of constant height $n < \infty$.

Under these circumstances, the power operations on $E$-cohomology give
rise to the following result.
To formulate it, we will write $f^*\mb G$ for the image of a formal
group $\mb G$ along a ring homomorphism $f$.  

\begin{thm}[{\cite[Theorem 3.25]{ando-hopkins-strickland}}]
\label{thm:level-structures}
The ring $E^0$ with formal group $\mb G$ has descent data for level
structures, as follows.  The pair
\[
(E^0({B\Sigma_p})/I, \Spec(E^0(BC_p))\subseteq
\psi^*\mb{G})
\]
is the universal example of a pair $(S,H)$ consisting of an
$E^0$-algebra $f\co E^0\to S$ together with a subgroup $H\subseteq
f^*\mb{G}$ of rank $p$. Given such a pair $(S,H)$, one obtains a
classifying map of $E^0$-algebras $F:E^0({B\Sigma_p})/I\to S$.  The
resulting ring homomorphism $f^{(p)} = F\circ \psi\co E^0\to S$
comes with an isogeny $\psi^p_{\mb{CP}^\infty}\co f^* \mb G \to
(f^{(p)})^* \mb G$ with kernel $H$.  These fit into a commutative
diagram as follows:
\[
\xymatrix{
& E^0 \ar[r]^{f^{(p)}} \ar[ld]_f &
  S \ar[d]\\
  S \ar[r] & S/(p) \ar[r]^{{\rm Frob}} & S/(p)
}
\]
\end{thm}

\begin{rmk}
The reference \cite{ando-hopkins-strickland}, in order to guarantee
the existence of universal deformations, makes the blanket
assumption that $E^0$ is a complete local ring with {\em perfect}
residue field of characteristic $p > 0$.  In the course of this
paper we will have need for Theorem \ref{thm:level-structures}
in the case where the residue field is not perfect; we note that the
proof of the above theorem is a formal consequence of calculations
in $E$-cohomology and does not require a perfect residue field.
\end{rmk}

One particular case is when $E$ is the Lubin-Tate cohomology theory
associated to a formal group law $\mb G_0$ of height $n < \infty$ over
a perfect field $k$ of characteristic $p$
\cite{rezk-notes-on-hopkins-miller, goerss-hopkins}.  In this
  case, we have the following algebraic description of the descent
structure coming from the $E_\infty$-structure.

\begin{prop}[{\cite[Proposition 12.12]{ando-hopkins-strickland}}]
\label{prop:lubin-tate-descent}
If $E$ is a Lubin-Tate cohomology theory, the descent data for level
structures on $E^0$ provided by the (unique) $E_\infty$-structure of
Goerss and Hopkins coincides with the descent data provided by the
universal property of the Lubin-Tate ring.
\end{prop}

\subsection{Height 1 objects}\label{subsec:heightone}
In this section we discuss the power operations of the previous
section in the height $1$ case.  In this case, any formal group $\mb
G$ of constant height $1$ has a unique subgroup scheme of rank $p$
given by the kernel of multiplication-by-$p$, which we denote by $\mb
G[p] \subseteq \mb{G}$.  The ring $E^0(B\Sigma_p)/I$ classifying
subgroups of rank $p$ is therefore isomorphic to $E^0$, and the power
operation takes the form of a map $\psi^p\co E^0 \to E^0$ which is a
lift of Frobenius.

Conversely, suppose we are given a $p$-adic ring $E$ with formal group
$\mb G$ of constant height $1$, together with a lift of Frobenius
$\psi\co E \to E$.  Write $\overline{\mb G}$ for the mod-$p$
  reduction of $\mb G$ and $\overline{\mb G}^{(p)}$ for the
  pullback of $\overline{\mb G}$ along the Frobenius map.  In this case, rigidity of formal
tori says that the canonical homomorphism
\[
\mathit{Hom}_E(\mb G,\psi^* \mb G)\to \mathit{Hom}_{E/(p)}(\overline{\mb
  G},\overline{\mb G}^{(p)}),
\]
between abelian groups of homomorphisms of formal groups over the
indicated base rings, is a bijection.  This is easier to see on the
dual groups $\mb G^\vee$ and $(\psi^* \mb G)^\vee$, which are \'etale,
by using the fact that the reduction $E\to E/(p)$ induces an
equivalence of \'etale sites. In particular, there is a unique lift of
the relative Frobenius isogeny $\overline{\mb G}\to\overline{\mb
  G}^{(p)}$ over $E/(p)$ to an isogeny $\mb G\to\psi^*\mb G$ over $E$,
necessarily of degree $p$.

Now suppose $F$ is an $E_\infty$-ring spectrum
satisfying the assumptions stated at the beginning of
Section~\ref{sec:homogeneous}, and define
$E = L_{K(1)} F$.  The natural localization map $F \to E$
gives rise to a diagram of total power operations:
\[
\xymatrix{
B:=F^0 \ar[d] \ar[r]^{\psi^p} & F^0(B\Sigma_p)/I \ar[d]\\
\tilde{B}:=E^0 \ar[r]^{\psi^p} & E^0(B\Sigma_p)/I & E^0 \ar[l]_{\hspace{2pc}\cong}
}
\]
Naturality of the description of $\psi^p$ implies that the right-hand
vertical map can be expressed in terms of descent data.  Specifically,
the formal group $\mb{G}_B$ over $B$ determined by the complex
orientability of $F$ has a unique subgroup scheme $C$ of rank $p$
after being pulled back to $\tilde{B}=E^0$, and the
composite map $F^0(B\Sigma_p)/I \to E^0(B\Sigma_p)/I \cong E^0$
classifies this subgroup. In other words, there is an isomorphism
\begin{equation}\label{eq:toplevel}
(\psi^p)^*\mb{G}_{\tilde{B}}\cong\mb{G}_{\tilde{B}}/C
\end{equation}
over  $\tilde{B}=E^0$.

\section{Realization problems for generalized $\BPP{n}$}
\label{sec:realization}

In this section we describe a family of graded rings, together with
formal group laws, that share chromatic features with the formal group
law of $\BPP{n}$.  Our aim is to understand when these can be realized
as the homotopy of complex oriented $E_\infty$-ring spectra.  The
chromatic machinery described in this paper iteratively realizes
various algebras over these rings, and this section also provides some
technical tools for understanding the formal group data of a spectrum
based on knowledge of the formal group data of an algebra over it.

\begin{defn}
\label{defn:realization}
Fix a prime $p$ and an integer $n\ge 0$.
A {\em (generalized) $\BPP{n}$-realization problem} is a graded $\mb
  Z_{(p)}$-algebra $A$, equipped with a map $MU_* \to A$ classifying
a graded formal group law $\mb G$, such that
\begin{itemize}
\item $A$ is concentrated in nonnegative degrees,
\item $A$ is degreewise finitely generated free over $\mb Z_{(p)}$,
\item the sequence $p$, $v_1$, \ldots, $v_n$ in $A$ determined
  by $\mb G$ is a regular sequence, and 
\item the quotient of $A$ by the ideal $(p,v_1,\ldots,v_n)$
is isomorphic to $\mb F_p$ concentrated in
  degree zero.
\end{itemize}
\end{defn}

\begin{rmk}
This definition admits several equivalent formulations. 
One of these is that the composite map
\[
\mb Z_p[v_1,\ldots,v_n] \subseteq BP_* \to (MU_{(p)})_* \to A
\]
is an isomorphism.  Another is that $A$ is a quotient of $BP_*$
by a regular sequence given by lifts of the classes $v_k$ for $k >
  n$. The given definition depends only on the formal
group rather than the formal group law, and does not depend on choices
of lifts of the classes $v_m$, which are only uniquely determined mod
$(p,v_1,\ldots,v_{m-1})$.
\end{rmk}

\begin{defn}
Suppose $(A,\mb G)$ is a (generalized) $\BPP{n}$-realization
problem.  For any graded $A$-algebra $B$, a {\em solution} to this realization
problem {\em for $B$} is a complex oriented $E_\infty$-ring spectrum $T$,
equipped with an isomorphism of graded rings $T_* \to B$ which sends
the formal group of $T$ to the image of the formal group $\mb
G$ over $B$.

Two solutions $T$ and $T'$ are {\em equivalent} if there exists a weak
equivalence $T \to T'$ of $E_\infty$-ring spectra 
which, on homotopy groups, is compatible with the given
identifications of $B$ with $T_*$ and $T'_*$.
\end{defn}

\begin{defn}
Suppose $(A,\mb G)$ is a (generalized) $\BPP{n}$-realization
problem, and recall the consequent isomorphism $A \cong \mb
Z_{(p)}[v_1,\ldots,v_n]$.  For any graded $A$-algebra $B$ and $0
\leq m \leq n$, we define the ``$K(m)$-localization'' of $B$ to
be the $B$-algebra
\[
B_{K(m)} = \comp{(v_m^{-1} B)}_{(p,v_1,\ldots,v_{m-1})}.
\]
When $m=0$ we follow the usual convention that $v_0 = p$.  Here
the completion of a graded ring by a homogeneous ideal denotes
the completion in the category of graded rings.
\end{defn}

\begin{rmk}\label{rmk:ringcalculations}
The following are some cases relevant for the present paper.
Suppose $A$ is the graded ring $\mb Z_{(p)}[v_1,v_2]$ with
$|v_i|=2(p^i-1)$, and let $x = v_1^{p+1}v_2^{-1}$, of degree $0$.
In the following, $R\{ \cdots \}$ denotes a free $R$-module with
indicated generators.
\begin{align*}
A_{K(0)}&=\mb Q[v_1, v_2]&\mbox{ and }
(A_{K(0)})_0&=\mb Q.\\
A_{K(1)}&=\comp{\mb Z_p[x^{-1}]}_p[v_1^{\pm 1}]&\mbox{ and }
(A_{K(1)})_0&=\comp{\mb Z_p[ x^{-1}]}_p.\\
A_{K(2)}&=\mb Z_p\pow{x}[v_2^{\pm 1}]\cdot\{ 1,v_1,\ldots, v_1^p\}\\
&=\mb Z_p\pow{x}[v_2^{\pm 1},v_1]/(v_1^{p+1}-xv_2)&\mbox{ and }
(A_{K(2)})_0&=\mb Z_p\pow{x}.\\
A_{K(2),K(1)}&=\left( \comp{\mb Z_p\laur{x}}_{p}\right)[v_1^{\pm 1}]
\end{align*}

For the last, there are several possible ways to describe the
degree-$0$ part:
\begin{align*}
(A_{K(2),K(1)})_0
&= \comp{\left( \mb Z_p\pow{x}[x^{-1}]\right)}_p\\
&= \comp{\mb Z_p\laur{x}}_p\\
&=\left\{ \sum\limits_{n\in\mb Z} \alpha_nx^n\,\middle\vert\, \alpha_n\in \mb Z_p, \alpha_n\to 0\mbox{ as } n\to - \infty\right\}
\end{align*}
\end{rmk}

\begin{rmk}\label{rem:homotopyoflocalizations}
The interest in this algebraic construction is as follows.
If the generalized $\BPP{n}$-realization problem has a solution $R$,
one can check that
\[
\pi_*(L_{K(n)}R)\cong (\pi_* R)_{K(n)}
\]
as graded algebras over $\pi_*(R)\cong A$. The generalizations to
iterated application of $K(i)$-localization for $0\leq i\leq n$ are
also true. One may prove this result using \cite[Theorem
1.5.4]{hovey-chromaticsplitting} by first showing that $R$ satisfies
the telescope conjecture using \cite[Theorem 1.9]{hovey-vnelements},
or alternatively by using the results of \cite{greenlees_may}, in a
way generalizing the proof of \cite[Lemma
8.1]{behrens-construct}.
\end{rmk}

For the remainder of this paper we concentrate on
generalized $\BPP{2}$-realization problems, and simply refer to them
as {\em realization problems}.

The following two lemmas will prove convenient for lifting formal
group data to a homotopy commutative ring spectrum $R$ from a given
$R$-algebra $S$.

\begin{lem}
\label{lem:fgl-descent}
Suppose $f\co A \to B$ is a faithfully flat map of commutative rings
and $A$ is torsion-free.  If $\mb G$, $\mb G'$ are two formal group
laws over $A$ such that $f^* \mb G$ and $f^* \mb G'$ are strictly
isomorphic over $B$, then $\mb G$ and $\mb G'$ are strictly isomorphic
over $A$.
\end{lem}

\begin{proof}
Let $h(x) = x + \sum b_n x^{n+1}$ be a strict isomorphism between
$\mb G$ and $\mb G'$ with coefficients $b_n\in B$.  Consider the
ring $B \otimes_A B$, equipped with two unit maps $\eta_L, \eta_R\co
B \to B \otimes_A B$.  We define $g$ to be the common composite
$\eta_L \circ f = \eta_R \circ f$.  The power series $\eta_L^*(h)$
and $\eta_R^*(h)$ provide two strict isomorphisms $g^* \mb G
\rightrightarrows g^* \mb G'$.

The map $g$ is flat.  Hence $B \otimes_A B$ is torsion-free, and
embeds into its rationalization $(B \otimes_A B) \otimes \mb Q$.  Over
this ring the composite $\eta_L^*(h)^{-1} \circ \eta_R^*(h)$ is a
strict automorphism of $(g^* \mb G)_{\mb Q}$, which is isomorphic to
the additive group.  The only strict automorphism of the additive group
in characteristic zero is the identity, and hence $\eta_L^*(h) =
\eta_R^*(h)$.

The power series $h$ thus has coefficients in the equalizer of the
maps
\[
\eta_L,\eta_R\co B \rightrightarrows B \otimes_A B.
\]
However, faithfully flat descent implies that this equalizer is
precisely $A$.
\end{proof}

\begin{cor}
\label{cor:orientation-descent}
Suppose $f\co R \to S$ is a map of complex oriented, homotopy
commutative ring spectra such that $R_*$ and $S_*$ are concentrated in
even degrees, $R_*$ is torsion-free, and the induced map $f_*\co
R_* \to S_*$ is faithfully flat.  Suppose that, on homotopy groups,
the orientation $g\co MU \to S$ has a factorization $MU_* \to R_* \to
S_*$.  Then the orientation $g$ itself factors through $R$.
\end{cor}

\begin{proof}
We have a second orientation $g'\co MU \to R \stackrel{f}{\to} S$,
given by the composite of any orientation of $R$ with the map $f$.
The two induced maps $g_*,g'_*\co MU_* \to S_*$ both factor through
$R_*$ by assumption, and as they arise from two orientations $MU\to
S$ the two formal group laws classified by $g_*$ and $g'_*$ over
$S_*$ are strictly isomorphic.  Lemma~\ref{lem:fgl-descent} shows
that there is a strict isomorphism of these formal group laws over
$R_*$. The result follows: any choice of orientation $MU \to R$,
with induced formal group law $\mb G$, establishes a bijection
between homotopy classes of ring spectrum maps $MU \to R$ and strict
isomorphisms with domain $\mb G$.
\end{proof}

\begin{lem}
\label{lem:fgl-pullback}
Suppose that we have a homotopy pullback diagram
\[
\xymatrix{
R \ar[r] \ar[d] & R_1 \ar[d]\\
R_2 \ar[r] & S
}
\]
of homotopy commutative ring spectra such that $S$ is complex
orientable and $S_*$ is concentrated in even degrees.  Then the set of
complex orientations of $R$ maps isomorphically to the pullback of the
sets of complex orientations of $R_1$ and $R_2$ over the set of
complex orientations of $S$.
\end{lem}

\begin{proof}
Recall that a complex orientation of $R$ is the same as the
homotopy class of a map of ring spectra $MU\to R$.  By
\cite[Proposition 2.16]{hovey-strickland} and the classical
computation $MU_*MU\cong MU_*[b_i]$ with $|b_i|=2i$ $(i\ge 1)$, we
have isomorphisms
\[
  S_* MU \cong S_*[b_i]\text{ and }S^* MU \cong \Hom_{S_*}(S_* MU, S_*).
\]
In particular, both groups are concentrated in even degrees by the
assumption on $S_*$, and the Mayer-Vietoris sequence degenerates to
a pullback diagram
\[
\xymatrix{
[MU,R] \ar[r] \ar[d] & [MU,R_1] \ar[d] \\
[MU,R_2] \ar[r] & [MU,S].
}
\]
The same argument applies with $\mb S$ and $MU \smsh{} MU$ in place of
$MU$.  Therefore, restricting to the subsets of maps commuting with
the unit map and multiplication map, we obtain a pullback diagram of
sets of homotopy classes of ring spectrum maps, as desired.
\end{proof}

The following lemma applies Cooke's obstruction theory
\cite{cooke} to lift group actions and equivariant maps from the
homotopy category of $E_\infty$-ring spectra to honest actions.
Recall all spectra are implicitly assumed to be $p$-local for a fixed
prime $p$.  These techniques have been previously applied to the
construction of group actions on $E_\infty$-ring spectra
\cite{behrens-construct,szymik-localk3}.

\begin{lem}
\label{lem:groupactions}
Let $G$ be a finite group of order prime to $p$.
  \begin{enumerate}
  \item Suppose that $R$ is an $E_\infty$-ring spectrum.
    Any action of $G$ on $R$ in the homotopy category of
    $E_\infty$-ring spectra has a unique lift, up to weak equivalence,
    to an action of $G$ on $R$ in the category of $E_\infty$-ring spectra.
  \item Let $R$ and $S$ be $E_\infty$-ring spectra, equipped
    with actions of $G$ in the category of $E_\infty$-ring spectra.
    Suppose that $\pi_1 \Map_{E_\infty}(R,S)$ is abelian for all
      choices of basepoint. Then the natural map
\[
\pi_0 \Map^G_{E_\infty}(R,S) \to
(\pi_0 \Map_{E_\infty}(R,S))^G,
\]
from the connected components of the derived mapping space of
$G$-equivariant $E_\infty$-ring spectra maps to $G$-invariant
components of the derived $E_\infty$-mapping space, is an
isomorphism.
  \end{enumerate}
\end{lem}

\begin{proof}
  Let $f\co R \to S$ be a map of $E_\infty$-ring spectra.  Without
  loss of generality, assume that $R$ and $S$ are
  fibrant-cofibrant and that the map $f$ is a cofibration.
  We define $\hAut_{E_\infty}(R) \subseteq \Map_{E_\infty}(R,R)$ to be
  the topological submonoid consisting of homotopy equivalences; this
  is a union of path components.  This allows us to introduce a monoid
  of homotopy automorphisms of the map $f$ by a pullback diagram of
  mapping spaces (which is a homotopy pullback) as follows:
\[
\xymatrix{
  \hAut_{E_\infty}(f) \ar[r] \ar[d] &
  \hAut_{E_\infty}(R) \ar[d]^{f_*}\\
  \hAut_{E_\infty}(S) \ar[r]_{f^*\hspace{1pc}}&
  \Map_{E_\infty}(R,S)
}
\]
We will show that, for any map $f\co R \to S$ of $E_\infty$-ring
spectra and any homomorphism
\[
G \to \pi_0 \hAut_{E_\infty}(R) \times_{\pi_0 \Map_{E_\infty}(R,S)}
  \pi_0 \hAut_{E_\infty}(S),
\]
the map $f$ is weakly equivalent to a unique replacement map $R'
\to S'$ of $G$-equivariant $E_\infty$-ring spectra.  Taking $f$ to be
the identity map of $R$ proves the first statement of the lemma;
taking $f$ to be arbitrary proves the second.

Having fixed a base map $f$, the space of based maps $S^n \to
  \Map_{E_\infty}(R,S)$ is identified by adjunction with the space of
  $E_\infty$-maps which are lifts in the diagram
\[
\xymatrix{
& S^{S^n} \ar[d] \\
R \ar[r]^{f} \ar@{.>}[ur] & S.
}
\]
For $n > 0$, the multiplication-by-$|G|$ map on $S^n$ induces an
$E_\infty$ self-equivalence of $S^{S^n}$.  Passing to $\pi_0$, we find
that $g \mapsto g^{|G|}$ is a bijective self-map of
$\pi_n(\Map_{E_\infty}(R,S),f)$. Therefore, the Mayer-Vietoris
sequence implies that $\pi_0 \hAut_{E_\infty} (f)$ is an extension
of $\pi_0 \hAut_{E_\infty}(R) \times_{\pi_0 \Map_{E_\infty}(R,S)}
\pi_0 \hAut_{E_\infty}(S)$ by an abelian group in which
multiplication by $|G|$ is an isomorphism.  This implies that the
given homomorphism from $G$ to the pullback on $\pi_0$, representing
actions of $G$ on $R$ and $S$ commuting with $f$ in the homotopy
category, lifts uniquely to a homomorphism
\[
G \to \pi_0 \hAut_{E_\infty}(f).
\]
Upon applying classifying spaces of the associated monoids, we obtain
a diagram of spaces as follows:
\[
\xymatrix{
&B\hAut_{E_\infty}(f) \ar@{>>}[d]\\
BG \ar@{..>}[ur] \ar[r] & B\pi_0\hAut_{E_\infty}(f)
}
\]
We now prove that there is a unique lift in the above
diagram up to homotopy.
The standard Postnikov-tower obstruction theory for lifting a map 
$Z \to X$ along a map $Y \to X$ gives an iterative sequence of
obstructions in $H^s(Z; \pi_s(X,Y))$, with choices of
lifts determined by classes in $H^s(Z; \pi_{s+1}(X,Y))$.  In the present
case, the homotopy groups of $B\hAut_{E_\infty}(f)$ relative to
$B\pi_0\hAut_{E_\infty}(f)$ are $p$-local in degrees above $1$, and so
the cohomology of $BG$ vanishes; this implies the
existence of a unique lift of $f$ to an equivariant map.

Given this lift, we can apply Cooke's
techniques to replace the map $f\co R \to S$ by a weakly 
equivalent map in the homotopy category such that $G$ acts on $R$ and
$S$ in a way making the map $f$ $G$-equivariant, as desired.
\end{proof}

\section{$K(2)$-local realization}
\label{sec:k2-localization}

In this section we apply the Goerss-Hopkins-Miller theorem to show
that any realization problem, as in
  Definition~\ref{defn:realization}, has an essentially unique
$K(2)$-local solution.

Let $\zeta$ be a primitive $(p^2-1)$'st root of unity in $\mb
F_{p^2}$.  We will also use $\zeta$ to denote its Teichm\"uller lift
in the $p$-typical Witt ring $\mb W(\mb F_{p^2}) \cong \mb
Z_p[\zeta]$.  Let $\sigma$ denote the nontrivial element of the Galois
group $Gal(\mb F_{p^2}/\mb F_p)$.

\begin{defn}
\label{def:associated-lubintate}
Suppose $(A,\mb G)$ is a realization problem, and let $v \in
A_{2p^2-2}$ be any lift of the canonical class $v_2$ (which is only
well-defined mod $(p,v_1)$).  The {\em associated Lubin-Tate
ring} is defined to be the graded $A$-algebra
\[
B = A_{K(2)}[\zeta, u]/(u^{p^2-1} - v)
\]
Here $|u_1|=0$ and $|u|=2$.

The {\em associated group} $G$ is defined to be the semidirect product
\[
(\mb F_{p^2}^\times) \rtimes Gal(\mb F_{p^2}/\mb F_p)
\]
formed with respect to the action of the Galois group on $\mb
F_{p^2}^\times$.  The group $G$ acts on $\zeta$ through the
quotient $Gal(\mb F_{p^2}/\mb F_p)$, and any element $x \in \mb
F_{p^2}^\times\subseteq G$ fixes $\zeta$ and acts on $u$ by
multiplication by the Teichmuller lift of $x$.
\end{defn}

\begin{rmk}
If we choose a lift of $v_1$-class to $A$, we obtain by Remark
  \ref{rmk:ringcalculations} an isomorphism
\[
B  \cong \mb W(\mb F_{p^2})\pow{u_1}[u^{\pm 1}], 
\]
where the $A$-algebra structure is determined by $v_1 \mapsto u_1
u^{p-1}$, $v_2 \mapsto u^{p^2-1}$.  The group $G$ then acts on the
whole graded $\mb Z_p$-algebra $B$ as follows:
\begin{alignat*}{2}
\zeta \cdot u_1 &= \zeta^{p-1} u_1&\hspace{2pc}
\sigma \cdot u_1 &= u_1\\
\zeta \cdot u &= \zeta^{-1} u&
\sigma \cdot u &= u
\end{alignat*}
\end{rmk}

\begin{rmk}
\label{rmk:lubintate}
To justify referring to $B$ as a Lubin-Tate ring, we observe the
following.  The action of $G$ preserves the graded formal group law.
By construction, the ungraded quotient ring $B/(u-1)$ is the complete
local ring $\mb W(\mb F_{p^2})\pow{u_1}$ with maximal ideal $(p,u_1)$.
Modulo $p$, the image of the formal group law $\mb G$ of $B$ has
$p$-series congruent to $u_1 x^{p}$ plus higher order terms, and
modulo $(p,u_1)$ it has $p$-series congruent to $x^{p^2}$ plus higher
order terms.  Hence the image of $\mb G$ over $B$ is of exact
height $2$, carries a canonical choice of invariant differential, and
is a universal deformation of the reduction of this data to $\mb
F_{p^2}$ \cite[Proposition 1.1]{lubintate}.
\end{rmk}

\begin{lem}
\label{lem:lubintate-galois}
Suppose $(A,\mb G)$ is a realization problem, $v\in A_{2p^2-2}$ a lift
of $v_2$, and $B$ the associated Lubin-Tate ring with associated group
$G$. Then $B$ is a Galois extension of $A_{K(2)}$ with Galois group
$G$, in the sense of \cite[Definition 2.3.1]{rognesgalois}.
\end{lem}

\begin{proof}
The ring $B$ is the extension of $A_{K(2)}$ obtained by
adjoining the primitive $(p^2-1)$'st root of unity $\zeta$ and the
element $u$ satisfying $u^{p^2-1} = v$.  The underlying module
of $B$ is finitely generated free as a module over $A_{K(2)}$, with
basis given by the images of the element $\sum_{i=0}^{p^2-2} \zeta
u^i$ under the action of $G$.  Therefore,
\[
B \cong A_{K(2)} \otimes \mb Z[G]
\]
as a module with $G$-action.  The fixed subring is
$A_{K(2)}$, and the map
\[
B \otimes_{A_{K(2)}} B \to \prod_G B
\]
is an isomorphism, as desired.
\end{proof}

\begin{thm}
\label{thm:k2local-exists}
Suppose $(A,\mb G)$ is a realization problem.  Then there exists a
solution to the realization problem for the $A$-algebra $A_{K(2)}$ by
a $K(2)$-local $E_\infty$-ring spectrum $R_{K(2)}$, and any two
solutions to the realization problem for $A_{K(2)}$ are equivalent.
\end{thm}

\begin{proof}
We first show that a solution exists.  Let $v \in A_{2p^2-2}$ be any
lift of the canonical class $v_2$ as in
Definition~\ref{def:associated-lubintate}.  Let $B$ be the associated
Lubin-Tate ring and $G$ the associated group.

Remark~\ref{rmk:lubintate}, together with the Goerss-Hopkins-Miller
theorem \cite[Corollary 7.6]{goerss-hopkins-published}, imply that
there exists a Lubin-Tate spectrum $E$ with $E_* \cong B$ in a
manner preserving the associated formal group laws: the formal
group law of $B$ lifts to an orientation $MU \to E$.  This
spectrum $E$ has an essentially unique structure as an $E_\infty$-ring
spectrum, and the group $G$ acts on $E$ by maps of $E_\infty$-ring
spectra.  We therefore have a homotopy fixed point spectrum $E^{hG}$
which is an $E_\infty$-ring spectrum.  We define $R_{K(2)}$ to be
$E^{hG}$, and remark that $R_{K(2)}$ is $K(2)$-local since $E$
is. We will now show that it is a solution to the realization
problem for $A_{K(2)}$.

Lemma~\ref{lem:lubintate-galois} implies that the fixed subring of
$E_*$ under the action of $G$ is isomorphic to $A_{K(2)}$, and
that the homotopy fixed point spectral sequence
\[
H^s(G; \pi_t E) \Rightarrow \pi_{t-s} R_{K(2)}
\]
degenerates to an isomorphism
\[
\pi_*(R_{K(2)}) \cong A_{K(2)}.
\]

The orientation $MU_*\to E_*$ factors through $\pi_*(R_{K(2)}) \cong
A_{K(2)}$ because the graded formal group law of $E_*$ arises by
base change from $A$ (and hence from $A_{K(2)}$).  Therefore,
Corollary~\ref{cor:orientation-descent} implies that there exists
a corresponding lift of the orientation to an orientation $MU \to
R_{K(2)}$.  We thus find that $R_{K(2)}$ solves the realization
problem for $A_{K(2)}$.

We now show that any other solution is equivalent, and in
particular that the solution is independent of $v$.  Suppose $T$ is
another solution to the realization problem for $A_{K(2)}$.  By
Lemma~\ref{lem:lubintate-galois}, the map $T_* \cong A_{K(2)} \to B$
is a Galois extension with Galois group $G$, and in particular
\'etale.  Therefore, the results of \cite[Section
2]{baker-richter-algebraicgalois} imply that there exists a unique
$E_\infty$ $T$-algebra $E'$ inducing this map on homotopy groups.  As
the map $T \to E'$ is Galois on homotopy groups, we have a weak
equivalence
\[
E' \smsh{T} E' \to \prod_{g \in G} E'.
\]
Rognes' work on idempotents \cite[Section~10]{rognesgalois} shows
that $E' \smsh{T} E' \simeq \prod_G E'$ as a left $E'$-algebra, and
that the space of $T$-algebra maps $E' \to E'$, or equivalently the
space of $E'$-algebra maps $E' \smsh{T} E' \to E'$, is homotopically
discrete and equivalent to $G$.  Therefore, $E'$ may be replaced by an
equivalent $T$-algebra with a genuine action of $G$.

Moreover, $E'$ is a Lubin-Tate spectrum, and the Goerss-Hopkins-Miller
theorem implies that the $G$-equivariant identification of $\pi_* E$ with
$\pi_* E'$ lifts to a $G$-equivariant equivalence of $E_\infty$-ring
spectra $E \to E'$.  The induced equivalence of homotopy fixed point
spectra $E^{hG} \to (E')^{hG}$ is then an equivalence of
$E_\infty$-ring spectra $R_{K(2)} \to T$ respecting the identification
on the level of underlying homotopy groups. 
By definition, this means the two solutions $R_{K(2)}$ and $T$ are equivalent, as desired.
\end{proof}

\section{$K(1)$-local obstruction theory}
\label{sec:k1-local-obstruction}

The work in this section closely follows \cite[Section
7]{behrens-construct}, which in turn follows \cite[Section
2.2]{goerss-hopkins}.
For background on $K(1)$-local $E_\infty$-ring spectra the 
reader may consult \cite{hopkins-k1-local} and \cite{laures-splitting}.

The aim is to explain in some generality the construction of 
$K(1)$-local $E_\infty$-ring spectra from purely algebraic data.
We will always be working locally at some fixed prime $p$.

Section \ref{sec:knlocalanss} records some basic properties of the 
$K(n)$-local Adams-Novikov spectral sequence.

Section \ref{subsect:orientedtheories} describes the
structure supported by the $K$-theory of a $K(1)$-local ring spectrum.
The reader should consult 
\cite[Section 7.4]{goerss-moduli-stack} for a more conceptional
explanation of some of these results.

Section \ref{subsect:periodicity} addresses a slightly technical issue
which can best be explained in the context of realizing families of
Landweber exact theories as in \cite{goerss-landweber-families}.  We
need to construct realizations not only over schemes of the
form $\Spec(A)$, but slightly more generally over quotients of affine
schemes by an action of $\mb Z/(p-1)$.  To do so, we will need to
rigidify the formal group that appears, and this is done by means
of the Igusa tower.

The final section \ref{subsect:obstruction} contains the key
constructive result in the $K(1)$-local setting (Theorem
\ref{thm:realizability}) which produces the desired $K(1)$-local
$E_\infty$-ring spectrum and its chromatic attaching map.

We will write $K$ for the complex $K$-theory spectrum.
Following \cite{hovey-etheory}, we write $K^\vee_*(X)$ for $\pi_*
L_{K(1)}(K \smsh{} X)$; this is isomorphic to $\pi_* \holim (K/p^n \smsh{}
X)$ for $K/p^n$ denoting the smash product of $K$ with a mod $p^n$
Moore spectrum.  We write $\comp{K}_p$ for $L_{K(1)} K$ itself.

\subsection{The $K(n)$-local Adams-Novikov spectral sequence}\label{sec:knlocalanss}

\begin{thm}
\label{thm:kn-anss} 
Recall that a prime $p$ is fixed, and choose an integer $n\ge
1$. Let $E$ denote the standard Lubin-Tate spectrum of height $n$,
$K(n)$ the corresponding Morava $K$-theory, and $\mb G_n$ the extended
Morava stabilizer group. For every $K(n)$-local spectrum $X$, there is
a $K(n)$-localized $E$-based Adams-Novikov spectral sequence,
converging strongly to the homotopy of $X$.  If $E_*X$ is a flat
$E_*$-module, this spectral sequence is of the form
\[
E_2^{s,t}=H^s_c(\mb G_n; E^\vee_t X) \Rightarrow \pi_{t-s} X.
\]
Here $E^\vee_*X = \pi_*L_{K(n)}(E\wedge X)$ is the Morava module of
$X$.  This spectral sequence is natural in $X$.
\end{thm}

\begin{rmk} 
In this paper, we will only use the case $n=1$.  Here we
have $\mb G_1=\mb Z_p^\times$, $E=\comp{K}_p$ is $p$-adic
$K$-theory, and $E_*X$ is flat if and only if it is $p$-torsion
free.
\end{rmk}

\begin{proof}
The spectral sequence is constructed in \cite[Appendix
A]{devinatzhopkinshomotopyfixedpoint} (with $Z=S^0$) and the
convergence established in \cite[Proposition
A.3]{devinatzhopkinshomotopyfixedpoint}.  We need to determine the
$E_2$-term, which by \cite[Remark
A.9]{devinatzhopkinshomotopyfixedpoint} is the cohomology of the
complex associated to the cosimplicial abelian group
\[
\pi_*L_{K(n)}(E^{(\bullet +1)}\wedge X).
\]
By Landweber exactness of $E$ and our flatness assumption, we know
that the $E_*$-module
\[
\pi_*(E^{(\bullet +1)}\wedge X)\cong
\underbrace{E_*E\otimes_{E_*}\cdots\otimes_{E_*}E_*E}_{\bullet\mbox{
times }}\otimes_{E_*} E_*X
\]
is flat as a left $E_*$-module. Application of \cite[Theorem
3.1]{hovey-somespectral} with $M=E^{(\bullet +1)}\wedge X$, along
with Morava's fundamental computation (\cite[Proposition
2.2]{devinatzhopkinshomotopyfixedpoint}, see in particular page 9),
shows that
\[
\pi_*L_{K(n)}(E^{(\bullet +1)}\wedge X)\cong \mathrm{Map}_c(\mb
  (\mb G_n)^\bullet,E_*) \widehat{\otimes}_{E_*} E^\vee_*X.
\]
In addition, we conclude that $E^\vee_*X \cong
\comp{(E_*X)}_\mathfrak{m}$ for $\mathfrak{m}\subseteq E_*$ the unique
homogeneous maximal ideal. Because $\mb G_n$ is profinite and
$E^\vee_*X$ is $\mathfrak{m}$-adically complete, the canonical
map
\[
\mathrm{Map}_c((\mb G_n)^\bullet,E_*) \widehat{\otimes}_{E_*}E^\vee_*X\to
  \mathrm{Map}_c((\mb G_n)^\bullet,E^\vee_*X)
\]
is an isomorphism.  This identifies the $E_1$-term as the standard
complex computing continous group cohomology, as claimed.
\end{proof}

\begin{rmk}
Further results on the identification of the $E_2$-term of this
spectral sequence may be found in
\cite[Theorem~10.2]{davis-continuousaction} and \cite{davis-torii}.
\end{rmk}

\subsection{$K(1)$-local complex oriented theories}\label{subsect:orientedtheories}

Denote by $MUP$ the $2$-periodic version of the complex bordism
spectrum $MU$. 

\begin{lem}
\label{lem:p-complete}
Suppose $F$ is a complex oriented, homotopy commutative ring spectrum
 such that $F_*$ is $p$-torsion free.  Then there is a natural
 isomorphism
\[
\comp{(K_* \otimes_{MU_*} MU_* MU \otimes_{MU_*} F_*)}_p \to K^\vee_* F.
\]
If $F$ is also even-periodic and equipped with an
  $MUP$-orientation, then there is also a natural isomorphism
\[
\comp{(K_0 \otimes_{MUP_0} MUP_0 MUP \otimes_{MUP_0} F_0)}_p \to K^\vee_{0} F.
\]
\end{lem}
\begin{proof}
As $F$ is complex orientable, the natural map
\[
MU_* MU \otimes_{MU_*} F_* \to MU_* F
\]
is an isomorphism.  In particular, as $MU_* MU$ is a free right
$MU_*$-module, both sides are isomorphic to a direct sum of copies of
$F_*$ and hence are torsion-free. Smashing with a mod $p^n$-Moore
spectrum $M(p^n)$, we obtain an inverse system of isomorphisms
\[
(MU_* MU \otimes_{MU_*} F_*)/p^n \to MU_* (M(p^n) \smsh{} F)
\]
as $n$ ranges over the natural numbers. Applying $K_*\otimes_{MU_*}-$,
the Conner-Floyd isomorphism then implies that the natural map
\[
(K_* \otimes_{MU_*} MU_* MU \otimes_{MU_*} F_*)/p^n \to
(K/p^n)_* (F)
\]
is an isomorphism.  Taking the inverse limit over $n$,
we obtain the desired result because the torsion-freeness of $K_*F$
implies
\[
K^\vee_*F\cong \comp{(K_*F)}_p\cong\lim\limits_n (K/p^n)_*(F).
\]
If $F$ is also even-periodic, the above proof then carries through
with $MU$ replaced by $MUP$.  Taking the term of degree zero gives
  the desired result.
\end{proof}

We refer to \cite[Examples 2.6 and 2.9]{stricklandformal}
for the universal properties of the two canonical ring homomorphisms
$MUP_0 \to MUP_0 MUP$ employed in the proof of the following lemma.

\begin{lem}
\label{lem:indgalois}
Suppose $T$ is a $p$-adic ring equipped with a formal
group law $\mb G$, induced by a map $MUP_0 \to T$, such that the
induced formal group law on $T/p$ is of exact height $1$.  Define 
\[
V = \comp{(K_0 \otimes_{MUP_0} MUP_0 MUP \otimes_{MUP_0} T)}_p.
\]
Then:
\begin{enumerate}
\item The extension $T\to V$ has a canonical structure of an ind-Galois
extension with Galois group $\mb Z_p^\times$.  In particular, $V$
is the $p$-adic completion of a filtered direct system of finite
Galois extensions of $T$.
\label{item:canonicalgalstructure}
\item For all continuous characters
$\chi\co \mb Z^\times_p\to\mb Z^\times_p$, the twisted $\mb
Z_p^\times$-module $V(\chi)$ is cohomologically trivial: for all
closed subgroups $H < \mb Z_p^\times$, the continuous Galois
cohomology $H^s_c(H, V(\chi))$ vanishes for $s > 0$.
\label{item:charactertwist}
\item If $T$ is the $p$-adic completion of a smooth
$\mb Z_p$-algebra, then $V$ is the $p$-adic completion of an
ind-smooth $\mb Z_p$-algebra: $V$ is the $p$-adic completion of a
filtered direct system of smooth $\mb Z_p$-algebras.
\label{item:indsmoothness}
\end{enumerate}
\end{lem}

\begin{rmk}
A mod-$p$ version of part 1 of this result, and a generalization
to higher height, appear in \cite[Theorem~A.1 and
  Theorem~A.3]{baker-ellipticadams}.
\end{rmk}

\begin{proof}
The universal properties of $MUP_0$ and $MUP_0 MUP$ imply that the ring
map $T \to V$ is initial among continuous maps of $p$-adic rings
$f\co T \to S$ equipped with an isomorphism $f^* \mb G \to 
\widehat{\mb G}_m$.

Consequently, $V$ is the $p$-completion of the directed limit of the
$T$-algebras $V_n$, where $T \to V_n$ is initial among continuous
maps of $p$-adic rings $f\co T \to S$ equipped with an isomorphism
$f^* \mb G[p^n] \to \mu_{p^n} =\widehat{\mb G}_m[p^n]$.

By Cartier duality \cite[Section 1.2]{tate-pdivisible}, these are
equivalent to isomorphisms $\mb Z/p^n \to f^*\mb G^\vee[p^n]$
between the dual groups, or equivalently elements of $\mb G^\vee$
of order exactly $p^n$.  As $\mb G$ is of height $1$ and dimension
$1$, the group scheme $\mb G^\vee[p^n]$ is \'etale-locally isomorphic
to $\mb Z/p^n$.  Therefore, the subscheme $\Spf(V_n)$ of generators is
a principal $(\mb Z/p^n)^\times$-torsor over $\Spf(T)$.  The maps
$\Spf(V_{n+1}) \to \Spf(V_n)$ are the restrictions of the
multiplication-by-$p$ maps $\mb G^\vee[p^{n+1}] \to \mb G^\vee[p^n]$,
and therefore the pro-system $\Spf(V_n)$ assembles into a principal
torsor for the constant pro-group scheme $\mb Z_p^\times$.

However, by definition this is the same as the map $T \to V$ being
an ind-Galois extension of $p$-adic rings with Galois group $\mb
Z_p^\times$, establishing item~\ref{item:canonicalgalstructure}.

Now fix a continuous character $\chi\co \mb Z_p^\times\to\mb
Z_p^\times$.  Recall that the twisted $\mb Z_p^\times$-module
$V(\chi)=V\otimes_{\mb Z_p}\mb Z_p(\chi)$ has underlying $\mb
Z_p$-module $V$ but $\mb Z_p^\times$-action given by
\[
\alpha\cdot v=\chi(\alpha)(\alpha v)
\]
for all $\alpha\in\mb Z_p^\times,v\in V$.

Arguing as in \cite[Proposition 2.1 and p. 259]{tate}, one reduces the
proof of cohomological triviality for $V(\chi)$ to that for the
module \[\overline{V(\chi)}:=V(\chi)/pV(\chi).\]

If $\chi$ is the trivial character, the result is then immediate from
$V/pV$ being a $\mb Z_p^\times$-Galois extension of $T/p$. For general
$\chi$, we have $\overline{V(\chi)}=(V/pV)(\overline{\chi})$ where
  $\overline{\chi}$ is the composite
\[
\overline{\chi}\co \mb Z_p^\times\stackrel{\chi}{\longrightarrow}\mb
Z_p^\times\longrightarrow (\mb Z/p)^\times.
\]

For the subgroup $U = H \cap (1+p\mb Z_p)\subseteq \mb Z_p^\times$ we
have $\overline{V(\chi)}\cong V/pV$ as $U$-modules, and this is
cohomologically trivial over $U$ by the above. One concludes
cohomological triviality over $H$ by a straightforward
application of the Lyndon-Hochschild-Serre spectral sequence to the
exact sequence $1 \to U \to H \to H/U \to 1$ because $H/U$ has order
prime to $p$.  This shows item~\ref{item:charactertwist}.

To see item~\ref{item:indsmoothness}, suppose in addition that $T$ is
the $p$-adic completion of a smooth $\mb Z_p$-algebra.
For every $n\ge 1$, the fact that $T\subseteq V_n$ is
finite \'etale implies that $V_n$, too, is the $p$-adic
completion of a smooth $\mb Z_p$-algebra.
Hence
\[
\mb Z_p\to\bigcup_{n\ge 1} V_n
\]
is a filtered direct limit (in fact, an increasing union)
of $p$-adic completions of smooth $\mb Z_p$-algebras, 
and we argued previously that $V=\comp{(\cup_{n\ge 1}V_n)}_p$.
\end{proof}

\begin{cor}
\label{cor:smooth-ktheory}
Suppose $F$ is a $K(1)$-local, complex oriented, even-periodic
homotopy commutative ring spectrum with $F_*$ torsion-free such
that $F_0$ is the $p$-adic completion of a smooth $\mb
Z_p$-algebra. Then the ring $K^\vee_0F$ is the $p$-adic completion of
an ind-smooth $\mb Z_p$-algebra, as well as the $p$-adic completion of
an ind-\'etale $F_0$-algebra.
\end{cor}
\begin{proof}
By Lemma~\ref{lem:p-complete}, we have
\[
K^\vee_0 F \cong \comp{(\mb Z \otimes_{MUP_0} MUP_0 MUP
  \otimes_{MUP_0} F_0)}_p.
\]
We obtain a composite map
\[
\mb Z_p \to F_0 \to K^\vee_0 F,
\]
where $F_0$ is the $p$-adic completion of
a smooth $\mb Z_p$-algebra by assumption.  The result now follows from
Lemma~\ref{lem:indgalois} 
with $T=F_0$, remarking that $V = K^\vee_0 F$ by
  Lemma~\ref{lem:p-complete} and that the formal group law is of
strict height $1$ by the assumption that $F$ is
  $K(1)$-local.
\end{proof}

The main point of the next result is to weaken the assumption that $F_*$
be even-periodic. We remark that item~\ref{item:homotopytype} of it is
a special case of \cite[Theorem 1.3]{davis-torii}.

\begin{cor}
\label{cor:fixedpoints}
Suppose $F$ is a $K(1)$-local, complex oriented, homotopy commutative
ring spectrum with associated formal group law $\mb G/F_*$ such that
$F_*$ is $p$-torsion free and concentrated in even degrees.  Then the
following consequences hold.
\begin{enumerate}
\item The spectrum $F$ is $(2p-2)$-periodic.
\label{item:periodic}
\item The map $F_* \to (K^\vee_* F)^{\mb Z_p^\times}$ is an isomorphism.
\label{item:invariants}
\item For all closed subgroups $H$ of $\mb Z_p^\times$, the
    continuous Galois cohomology groups $H^s_c(H, K^\vee_t F)$
    vanish for all $t \in\mb Z$ and all $s > 0$.
\label{item:cohomology}
\item The natural map $F \to \left(L_{K(1)}(K \smsh{}
    F)\right)^{h\mb Z_p^\times}$ is a weak equivalence.
\label{item:homotopytype}
\end{enumerate}
\end{cor}
\begin{proof}
We note that the orientation of $F$ provides a map $L_{K(1)} MU \to
F$ of ring spectra.  However, the localization map $MU \to v_1^{-1} MU$ is a
$K(1)$-equivalence, and so $F_*$ is an algebra over $v_1^{-1} MU_*$.
This ring is therefore $(2p-2)$-periodic, as in
item~\ref{item:periodic}.
The choice of a periodicity generator $v_1\in F_{2p-2}$
lets us introduce $F_*^{per} = F_*/(v_1 - 1)$, an associated $\mb
Z/(2p-2)$-graded $p$-adic ring.

By Lemma~\ref{lem:p-complete}, the ring
\[
\left(K^\vee_* F\right)^{per} \cong \left(\comp{(K_* \otimes_{MU_*} MU_*MU
  \otimes_{MU_*} F_*)}_p\right)^{per} 
\]
parameterizes strict isomorphisms between the formal group law $\mb G$
and formal group laws of the form $x + y - \beta xy$, for $\beta$
invertible, over $p$-adic $F_*^{per}$-algebras.  As each such formal
group law admits the canonical nonstrict isomorphism $x \mapsto \beta
x$ to the multiplicative formal group law, this is equivalent to
parameterizing nonstrict isomorphisms between the multiplicative
formal group law and the formal group law of $F_*^{per}$.  As such, if
we identify the map $MU_* \to F_*^{per}$ with the corresponding map of
ungraded rings $MUP_0 \to F_*^{per}$, we obtain an isomorphism of
ungraded rings
\[
\left(K^\vee_* F\right)^{per} \cong \comp{(\mb Z \otimes_{MUP_0} MUP_0MUP
  \otimes_{MUP_0} F_*^{per})}_p.
\]
Observe $T:=F_*^{per}$ is a $p$-adic ring because $F$ is $K(1)$-local
and $F_*$ is torsion-free.  Applying Lemma~\ref{lem:indgalois} to
$T=F_*^{per}$ and extrapolating back to $F_*$, we find that $F_*$ is
the ring of $\mb Z_p^\times$-invariants in $K^\vee_* F$
(item~\ref{item:invariants} holds true) and that all twists of
the $\mb Z_p^\times$-module $K^\vee_0 F$ are cohomologically trivial.
Item~\ref{item:cohomology} now follows because for every $t\in\mb Z$,
the $\mb Z_p^\times$-module $K^\vee_tF$ is either zero or the twist of
$K^\vee_0F=V$ by the action of $\mb Z_p^\times$ on the $(t/2)$'th
power of the Bott element.  The final item~\ref{item:homotopytype}
follows from items~\ref{item:invariants} and \ref{item:cohomology} and
the Adams-Novikov spectral sequence in Theorem~\ref{thm:kn-anss}.
\end{proof}

\begin{rmk}
  We note that switching to $(2p-2)$-periodic rings in the previous
  proof is solely to avoid a technical issue. Namely, the
  underlying ring of a $p$-adic $\mb Z$-graded ring is not $p$-adic
  when considered as an ungraded ring, but the analogous
  assertion does hold true for $(2p-2)$-periodic rings.
\end{rmk}

\subsection{$K(1)$-local theories and periodicity}\label{subsect:periodicity}

In order to translate results about even-periodic cohomology theories
into results about general $K(1)$-local complex orientable
theories, we recall the following construction. For every $n\ge 1$, we
have an exact sequence
\[
0 \to (1 + p^n \mb Z_p)\to \mb Z_p^\times \to (\mb Z/p^n)^\times
\to 0
\]
of groups, and $\mb Z_p^\times$ acts by $E_\infty$-ring maps on the
$p$-adic $K$-theory spectrum $\comp K_p$.
\begin{defn}\label{defn:igusacoverofsphere}
For every $n\ge 1$, let 
\[
\mb S(p^n) = (\comp{K}_p)^{h(1+p^n\mb Z_p)}
\]
be the $K(1)$-local $E_\infty$-ring spectrum which is the
continuous homotopy fixed point object of $(1 + p^n \mb Z_p)$
acting on $\comp{K}_p$, constructed by Devinatz and Hopkins
\cite{devinatzhopkinshomotopyfixedpoint}.  This is equipped with an
action of $(\mb Z/p^n)^\times$.  If $F$ is a $K(1)$-local spectrum,
then we define $F(p^n)$ to be the $K(1)$-local smash product
\[
L_{K(1)} (F \smsh{} \mb S(p^n)),
\]
viewed as a $(\mb Z/p^n)^\times$-spectrum.  We refer to the resulting
tower of $K(1)$-local spectra
\[
F\to F(p)\to F(p^2)\to\cdots
\]
as {\em the Igusa tower of $F$}.
\end{defn}

The nomenclature in Definition \ref{defn:igusacoverofsphere} is
motivated by item~\ref{item:4} in the following result.

\begin{lem}
\label{lem:igusa}
Suppose $F$ is a $K(1)$-local, complex oriented,
homotopy commutative ring spectrum with associated formal group law
$\mb G/F_*$ such that $F_*$ is $p$-torsion free and concentrated in
even degrees, and let $n\ge 1$.
\begin{enumerate}
\item The spectrum $F(p^n)$ is an $F$-module spectrum.  If $F$ is an
  $E_\infty$-ring spectrum, then $F(p^n)$ is also canonically an
  $E_\infty$-ring spectrum and the group $(\mb Z/p^n)^\times$ acts on $F(p^n)$
  by maps of $E_\infty$ $F$-algebras.
\label{item:2}
\item The $F_*$-algebra $F(p^n)_*$ is the coordinate ring ${\cal O}_{\mb
    G^\vee[p^n] \setminus {\mb G}^\vee[p^{n-1}]}$ of the elements of exact
  order $p^n$ in the $p$-divisible group $\mb G^\vee$ which is
    Serre dual to $\mb G$.  (See \cite[Section 2.3]{tate-pdivisible}
    for duality of $p$-divisible groups.)
\label{item:4}
\item The canonical map $F \to F(p^n)^{h(\mb Z/p^n)^\times}$ is a
    homotopy equivalence.
\label{item:3}
\item The spectrum $F(p)$ is even-periodic, and has the homotopy type of
  $\bigvee_{k=0}^{p-2} \Sigma^{2k} F$.  Specifically, we have
    $F(p)_*\cong F_*[s]/(s^{p-1} - b)$ as graded $F_*$-algebras for
    a suitable unit $b \in F_{2 - 2p}$.
\label{item:5}
\end{enumerate}
\end{lem}

\begin{proof}
Item~\ref{item:2} follows from the fact that the group $\mb Z_p^\times$
acts on the $p$-adic $K$-theory spectrum by $E_\infty$-ring maps,
and therefore $\mb S(p^n)$ inherits a $(\mb Z/p^n)^\times$-equivariant
$E_\infty$-ring structure.

To establish item~\ref{item:4}, recall that we have just seen that the
ring $F(p^n)_*$ is the subring of $(1 + p^n\mb Z_p)$-invariants in the
graded ring $K^\vee_* F$.  As in the proof of
Corollary~\ref{cor:fixedpoints}, after periodification $K^\vee_*
F$ is the universal ring parameterizing nonstrict isomorphisms between
the multiplicative formal group law and $\mb G$ over $p$-adic
$F_*$-algebras. The isomorphism $\mb G[p^n] \to \mu_{p^n}$ obtained by
restricting the universal isomorphism $\mb G\cong \widehat{\mb G}_m$
over $K^\vee_* F$ is invariant under the action of $(1 + p^n \mb Z_p)$
and hence defined over $F(p^n)_*$.  Moreover, the ring parameterizing
isomorphisms $\mb G \to \widehat{\mb G}_m$ which restrict to this
given isomorphism is a principal $\mb (1 + p ^n\mb Z_p)$-torsor over
$F_*$.  The ring of invariants $F(p^n)_*$ is precisely the universal
$F_*$-algebra parameterizing isomorphisms $\mb G[p^n] \to \mu_{p^n}$
of finite flat group schemes, and the grading is recovered through the
action of nonstrict isomorphisms $x \mapsto ux$.

As above, isomorphisms $\mb G[p^n] \to \mu_{p^n}$ are in functorial
bijection with isomorphisms $\mb Z/p^n \to \mb G^\vee[p^n]$ by Cartier
duality \cite[Section 1.2]{tate-pdivisible}.  The finite group scheme
$\mb G^\vee[p^n]$ is \'etale-locally cyclic of order $p^n$, and thus
such an isomorphism is equivalent to a choice of a section of $\mb
G^\vee[p^n]\setminus \mb G^\vee[p^{n-1}]$.  Therefore, the map 
$F_* \to F(p^n)_*$ is a $(\mb Z/p^n)^\times$-Galois extension with
range isomorphic to the coordinate ring ${\cal O}_{\mb
G^\vee[p^n] \setminus \mb G^\vee[p^{n-1}]}$, implying
item~\ref{item:4}.  The fact that this extension is Galois implies
that the homotopy fixed point spectral
\[
H^s((\mb Z/p^n)^\times; F(p^n)_*) \Rightarrow \pi_* F(p^n)^{h(\mb Z/p^n)^\times}
\]
degenerates, and the map $F \to F(p^n)^{h(\mb Z/p^n)^\times}$ is a
weak equivalence, establishing item~\ref{item:3}.

For item~\ref{item:5}, we employ the work of Oort-Tate
\cite{oort-tate} classifying finite flat group schemes of rank $p$
in order to determine the structure of the $F_*$-algebra
$F(p)_*$. The coordinate ring
\[
{\cal O}_{\mb G[p]} \cong F_*\pow{t}/[p]_{\mb G}(t)
\]
is a finite free $F_*$-algebra of rank $p$ with basis $\{t^k | 0
  \leq k < p\}$.  The ring ${\cal O}_{\mb G[p]}$ has an action of
$\mb F_p^\times \subseteq \mb Z_p^\times$ given by $\zeta \cdot t =
[\zeta]_{\mb G}(t)$, and without loss of generality we may replace the
orientation class $t$ by a new generator $s$ of the power
  series ring such that $\zeta \cdot s = \zeta
s$. We note that the action of nonstrict isomorphisms places the
  element $s$ in degree $-2$.

The Oort-Tate classification implies that we have an isomorphism
\[
{\cal O}_{\mb G^\vee[p]} \cong \bigoplus_{k=0}^{p-1} I^{\otimes k}
\]
for $I$ an invertible $F_*$-module, dual to the submodule of
${\cal O}_{\mb G[p]}$ where $\mb F_p^\times \subseteq \mb Z_p^\times$
acts by scalar multiplication.  In particular, the grading implies
that $I$ is concentrated in degrees congruent to $-2$ mod $(2p-2)$,
so we must have $I$ free with generator $s$.  The 
multiplication is induced by the tensor product pairing $I^{\otimes k}
\otimes I^{\otimes l} \to I^{\otimes (k+l)}$ together with a map
$I^{\otimes p} \to I$.  The augmentation to $F_*$ sends $I$ to zero.

We find that the ring ${\cal O}_{\mb G^\vee[p]}$ is of the form
\[
F_*[s]/(s^p - b s),
\]
for some element $b \in F_*$, and the zero section corresponds to the
quotient by the ideal $(s)$.  As this scheme is \'etale, the
Jacobian criterion \cite[Chapter I, Example 3.4]{milne-etale-cohomology} 
implies that $b$ must be a unit in $F_*$.

Therefore, the coordinate ring of the complement of the zero section
is
\[
F(p)_* \cong F_*[s]/(s^{p-1} - b).
\]
As $b$ is invertible, this implies that $F(p)_*$ is $2$-periodic.
Moreover, the $F_*$-module $F(p)_*$ is free with basis $\{1, s,
\ldots, s^{p-2}\}$.  Choosing maps $S^{2k} \to F(p)$ representing
the elements $s^k$, the $F$-module structure from item~\ref{item:2}
gives us maps $F \smsh{} S^{2k} \to F(p)$ which together provide a
weak equivalence $\bigvee \Sigma^{2k} F \to F(p)$ of $F$-modules.
This completes the proof of item~\ref{item:5}.

%To establish item~\ref{item:3}, first note that in case
%$(p,n)=(2,1)$ we are looking at the identity map by
%construction. In the cases where $p\neq 2$ or $n \neq 1$, the
%  group $(1+p^n\mb Z_p)$ is a free pro-$p$-group on the generator
%$1+p^n$ and there is a homotopy fiber sequence
%\[
%\xymatrix{
%F \smsh{} \mb S(p^n) \ar[r]&
%F \smsh{} \comp K_p \ar[rr]^{F\wedge \psi^{1+p^n}}&&
% F \smsh{} \comp K_p.}
%\]
%Upon applying $K(1)$-localization, we find that $F(p^n)$ is homotopy
%equivalent to the homotopy fixed point spectrum of the action of $(1
%+ p^n \mb Z_p)$ on $L_{K(1)} (F \smsh{} K_p)$.  Applying
%Corollary~\ref{cor:fixedpoints}, (\ref{item:cohomology}), we have an
%isomorphism $F(p^n)_* \cong (K^\vee_* F)^{1 + p ^n \mb Z_p}$.
%Corollary \ref{cor:fixedpoints}, (\ref{item:homotopytype}) then
%gives a homotopy equivalence $F \simeq F(p^n)^{h(\mb
%Z/p^n)^\times}$. This proves item~\ref{item:3}.

%Item~\ref{item:3} follows from Corollary \ref{cor:fixedpoints},
%item~\ref{item:homotopytype} and the transitivity of homotopy fixed
%points \cite[Theorem 4]{devinatzhopkinshomotopyfixedpoint}.
\end{proof}

\begin{rmk}
Daniel Davis pointed out an error in a previous version which assumed
a transitivity result for homotopy fixed points that is not implied
by the work of Devinatz-Hopkins.  A proof of this result will appear
in forthcoming work of Davis and Torii.
\end{rmk}

\begin{cor}
\label{cor:fgl-lifting}
Suppose $F$ is a $K(1)$-local, complex orientable, homotopy commutative
ring spectrum such that $F_*$ is $p$-torsion free and concentrated in
even degrees.  If there exists a complex orientation $MU \to F(p)$
which, on homotopy groups, factors as $MU_* \to F_*
\subseteq F(p)_*$,
then there exists a complex orientation $MU \to F$ realizing the given
map $MU_* \to F_*$.
\end{cor}

\begin{proof}
Lemma~\ref{lem:igusa}, item~\ref{item:5} implies that $F(p)_*$ is
faithfully flat over $F_*$, and then Corollary
\ref{cor:orientation-descent} shows that the orientation of $F(p)$
lifts to $F$.
\end{proof}

\subsection{Obstruction theory}\label{subsect:obstruction}

For an $E_\infty$-ring spectrum $R$, the $p$-completed $K$-theory
$K^\vee_* R$ has the structure of a $p$-complete ring equipped with an
action of $\mb Z_p^\times$ (the Adams operations $\psi^n$ for $n \in
\mb Z_p^\times$) by $(\comp K_p)_*$-algebra automorphisms, together
with a power operation $\theta$ on $K^\vee_0 R$ extending the
operation of the same name on $R_0$ and commuting with the Adams
operations.  This operation satisfies $\psi^p(x) = x^p + p \theta(x)$.
By contrast with the references, we will refer to such structure as a
$\psi$-$\theta$-algebra, and reserve the term $\theta$-algebra for a
$p$-adic ring equipped with the operation $\theta$ alone (such as the
homotopy of a $K(1)$-local $\comp{K}_p$-algebra).  These power
operations were studied by McClure \cite[Chapter IX]{h-infty}.  See
\cite[Section 2.2]{goerss-hopkins} for more details and the definition
of the relevant version of Andr\'e-Quillen cohomology, denoted here by
$H^*_{\psth\alg}$.

We recall the following consequences of the Goerss-Hopkins obstruction
theory.
\begin{thm}
\label{thm:goerss-hopkins}
  \begin{enumerate}
  \item\label{item:existenceofk1local} 
    Given a graded $\psth$-algebra $A_*$, the obstructions to the
    existence of a $K(1)$-local $E_\infty$-ring spectrum $R$ for which
    there is an isomorphism
\[
K^\vee_* R \cong A_*
\]
    of graded $\psth$-algebras lie in
\[
H^s_{\psth\alg}(A_*/(\comp{K}_p)_*, \Omega^{s-2} A_*), s \geq 3.
\]
    The obstructions to uniqueness lie in
\[
H^s_{\psth\alg}(A_*/(\comp{K}_p)_*, \Omega^{s-1} A_*), s \geq 2.
\]
    \item Given
      $K(1)$-local $E_\infty$-ring spectra $R_1$ and $R_2$ such that
      $K^\vee_* R_2$ is $p$-complete, together with a map of 
      graded $\psth$-algebras
\[
f_*\co K^\vee_* R_1 \to K^\vee_* R_2,
\]
    the obstructions to the existence of a map $f\co R_1 \to R_2$ of
    $E_\infty$-ring spectra inducing $f_*$ on $p$-adic $K$-theory lie in
\[
H^s_{\psth\alg}(K^\vee_*R_1/(\comp{K}_p)_*, \Omega^{s-1} K^\vee_* R_2), s \geq 2.
\]
    The obstructions to uniqueness lie in
\[
H^s_{\psth\alg}(K^\vee_*R_1/(\comp{K}_p)_*, \Omega^{s} K^\vee_* R_2), s \geq 1.
\]
  \item Given $K(1)$-local $E_\infty$-ring spectra $R_1$ and $R_2$ and
    a fixed $E_\infty$-ring map $\phi\co R_1 \to R_2$ such that $K^\vee_*
    R_2$ is $p$-complete, there is a fringed second-quadrant
    spectral sequence with $E_2$-term
\[
E_2^{s,t} = 
\begin{cases}
\Hom_{\psth\alg}(K^\vee_* R_1, K^\vee_* R_2) &\text{for $(s,t)
  = (0,0)$},\\
H^s_{\psth\alg}(K^\vee_*R_1/(\comp{K}_p)_*, \Omega^{t}
K^\vee_* R_2) &\text{for $t > 0$,}
\end{cases}
\]
abutting to
\[
\pi_{t-s}({\Map_{E_\infty}}(R_1, R_2), \phi).
\]
  \end{enumerate}
\end{thm}

In order, the references for the above facts are as follows:
\begin{enumerate}
\item This is an application of \cite[Corollary
  5.9]{goerss-hopkins-published} with the basic homology theory $E$
  there being chosen as $p$-complete $K$-theory.  To identify the
  resulting obstruction groups as $\psi$-$\theta$-cohomology, one
  invokes \cite[Sections 2.2 and 2.4]{goerss-hopkins}.
\item \cite[Corollary 2.4.15]{goerss-hopkins}, with $Z=S^0$.
\item \cite[Theorem 2.4.14]{goerss-hopkins}, with $Z=S^0$.
\end{enumerate}

The next result isolates purely algebraic assumptions which will
  demonstrate vanishing of many relevant Andr\'e-Quillen
cohomology groups.

\begin{lem}
\label{lem:psth-vanishing}
Suppose that $A_* \to B_*$ is a map of graded, $p$-adic, even-periodic
  $\psth$-algebras concentrated in even degrees such that
  \begin{itemize}
  \item $A_{0}$ is the $p$-adic completion of an ind-\'etale extension
    of the fixed subring $(A_{0})^{\mb Z_p^\times}
\subseteq A_0$,
  \item $(A_0)^{\mb Z_p^\times}$ is the $p$-adic completion of a smooth $\mb
    Z_p$-algebra, and
  \item the continuous group cohomology $H^s_c(\mb Z_p^\times,
    \Omega^t B_*)$ vanishes for $s > 0$ and all $t\in \mb Z$.
  \end{itemize}
Then the $\psth$-algebra cohomology
\[
H^s_{\psth\alg}(A_*/(\comp{K}_p)_*,\Omega^t B_*)
\]
vanishes for $s \geq 2$ or $t$ odd.

\end{lem}

\begin{proof}
To aid readability, define $T_* = (A_*)^{\mb Z_p^\times}$.

Using the smoothness assumption, we can apply \cite[Equation
(2.4.9)]{goerss-hopkins}, with $M=\Omega^tB_*$, to see that
\[
H^s_{\psth\alg}(A_*/(\comp{K}_p)_* ,\Omega^t B_*)\cong
Ext^s_{Mod_{A_*}^{\psth}}(\Omega_{A_*/(\comp{K}_p)_* 
}, \Omega^t B_*).
\]
As $A_*$ is even-periodic, sending a graded module $M_*$ to
  $(M_0, M_1)$ is an equivalence of categories between the category of
  graded $A_*$-$\psth$ modules and the category of pairs of
  $A_0$-$\psth$-modules.  Therefore, there is an isomorphism of
$\Ext$ groups
\[
Ext^s_{Mod_{A_*}^{\psth}}(\Omega_{A_*/(\comp{K}_p)_*}, \Omega^t B_*) \cong
Ext^s_{Mod_{A_0}^{\psth}}(\Omega_{A_0/\mb Z_p}, (\Omega^t B)_0)
\]
because $A_1$ is trivial.

As $A_0$ is the $p$-adic completion of an ind-\'etale $T_0$-algebra,
there is an isomorphism of modules of K\"ahler differentials
$\Omega_{A_0/\mb Z_p} \cong A_0 \otimes_{T_0} \Omega_{T_0/\mb Z_p}$
that respects the operation $\theta$ and the operations $\psi^k$
(the latter operations acting trivially on $T_0$).  Since $A_0$
is flat over $T_0$, we get an isomorphism
\[
\Ext^s_{Mod_{A_0}^{\psth}}(\Omega_{A_0/\mb Z_p},(\Omega^t B)_0) \cong
  Ext^s_{Mod_{T_0}^{\psth}}(\Omega_{T_0/\mb Z_p}, (\Omega^t B)_0).
\]
The right-hand Ext-group is computed in the category of modules over
$T_0$ equipped with a continuous $\mb Z_p^\times$-action and an
operator $\theta$, and the $\mb Z_p^\times$-action is trivial on the
domain. This leads to a composite functor spectral sequence:
\[
E_2^{s,u}=Ext^s_{T_0[\theta]}(\Omega_{T_0/\mb Z_p}, H^u_c(\mb
Z_p^\times,(\Omega^t B)_0))\Rightarrow
Ext^{s+u}_{Mod_{T_0}^{\psth}}(\Omega_{T_0/\mb Z_p},(\Omega^t B)_0)
\]
By assumption, the group cohomology degenerates, and we obtain the
result
\[
H^s_{\psth\alg}(A_*/(\comp{K}_p)_*,\Omega^t B_*)\cong
Ext^s_{T_0[\theta]}(\Omega_{T_0/\mb Z_p}, (\Omega^t B)_0^{\mb Z_p^\times}). 
\]
By smoothness, $\Omega_{T_0/\mb Z_p}$ is a projective
$T_0$-module, and hence these Ext-groups vanish for $s\ge 2$. They
also vanish for $t$ odd because $B_*$ is concentrated in even degrees.
\end{proof}

Direct application of Lemma~\ref{lem:psth-vanishing} with $A_*=B_*$ and
Theorem~\ref{thm:goerss-hopkins},
item~\ref{item:existenceofk1local}, yields the following
existence and uniqueness result for $K(1)$-local $E_\infty$-ring
spectra.

\begin{cor}
\label{cor:psth-realization}
Suppose that $A_*$ is an even-periodic graded $p$-adic $\psth$-algebra
concentrated in even degrees such that 
  \begin{itemize}
  \item $A_0$ is the $p$-adic completion of an ind-\'etale extension
    of $(A_0)^{\mb Z_p^\times}$,
  \item $(A_0)^{\mb Z_p^\times}$ is the $p$-adic completion of a
    smooth $\mb Z_p$-algebra, and
  \item the continuous group cohomology $H^s_c(\mb Z_p^\times,
    \Omega^t A_*)$ vanishes for $s > 0$ and all $t\in\mb Z$.
  \end{itemize}
Then there exists a $K(1)$-local $E_\infty$-ring spectrum $R$ such that
$K^\vee_* R \cong A_*$ as $\psth$-algebras, and any two such are weakly
equivalent.
\end{cor}

\begin{rmk}
With the results from Section~\ref{subsect:orientedtheories},
the previous corollary establishes an equivalence between a 
homotopy category of certain $K(1)$-local $E_\infty$-ring spectra
and a certain class of $\psth$-algebras.
\end{rmk}

Next, we start to study $E_\infty$-maps through purely algebraic data.

\begin{prop}
\label{prop:psth-mapping}
Suppose that $F$ and $F'$ are $K(1)$-local, even-periodic, complex
oriented $E_\infty$-ring spectra such that the rings $F_0$ and
$F'_0$ are $p$-torsion free and such that $F_0$ is the
$p$-adic completion of a smooth $\mb Z_p$-algebra.  Then the
canonical map
\[
[f] \mapsto K_0^\vee(f)\co \pi_0 \Map_{E_\infty}(F,F') \to
Hom_{\psth\alg}(K_0^\vee F, K_0^\vee F'),
\]
from the connected components of the derived $E_\infty$-mapping
space to the set of $\psth$-algebra maps, is bijective.
The canonical map
\[
[f]\mapsto \pi_0(f)\co \pi_0\Map_{E_\infty}(F,F')\to
Hom_{\theta\mbox{-}alg}(\pi_0 F,\pi_0F')
\]
is injective.  For any basepoint $f$ of the mapping
  space, the fundamental group $\pi_1 (\Map_{E_\infty}(F,F'),f)$ is
  abelian.
\end{prop}

\begin{proof}
  Define $A_* = K^\vee_* F$ and $B_* = K^\vee_* F'$.  By
  Corollary~\ref{cor:smooth-ktheory}, the ring $A_0$ is the $p$-adic
  completion of an ind-\'etale $F_0$-algebra.  By
  Corollary~\ref{cor:fixedpoints}, items~\ref{item:invariants} and
  \ref{item:cohomology}, we have $F_0 = (A_0)^{\mb Z_p^\times}$ and
  $(F')_0 = (B_{0})^{\mb Z_p^\times}$.  In addition, the cohomology
  groups $H_c^s(H, \Omega^t A_*)$ and $H_c^s(H, \Omega^t B_*)$ vanish
  for $H$ closed in $\mb Z_p^\times$, $s > 0$ and all $t\in\mb Z$.
  Therefore, Lemma~\ref{lem:psth-vanishing} implies that for any given
  map of $\psth$-algebras $A_*\to B_*$, the $\psth$-algebra cohomology
  $H^s_{\psth\alg}(A_*/(\comp{K}_p)_*, \Omega^t B_*)$ vanishes for
  $s\ge 2$ or $t$ odd.

  By Theorem~\ref{thm:goerss-hopkins}, item 2 we therefore find
  that any map of $\psth$-algebras $\bar f\co A_* \to B_*$ has a lift
  to a map of $E_\infty$-ring spectra $f\co F \to F'$, unique up to
  homotopy. This is equivalent to the claimed bijectivity of the first
  displayed map in this proposition.  By
  Theorem~\ref{thm:goerss-hopkins}, item 3 the fundamental group
  of the mapping space is always abelian.

To address the injectivity of the map given by evaluating
$\pi_0$, assume $f,g\co F\to F'$ are $E_\infty$-maps with
$\pi_0(f)=\pi_0(g)$. Then the functoriality of the Igusa tower
implies that $K_0^\vee(f)= K_0^\vee(g)$. Hence $f$ and $g$ are
homotopic as $E_\infty$-maps by the first part of the proposition.
\end{proof}

\begin{cor}
\label{cor:psth-action}
Let $G$ be a finite group of order prime to $p$.  Suppose that
  $F$ and $F'$ are $K(1)$-local, even-periodic, complex oriented 
  $E_\infty$-ring spectra, that the rings $F_0$ and 
  $F'_0$ are $p$-torsion free, and that $F_0$ is the $p$-adic
  completion of a smooth $\mb Z_p$-algebra.
  \begin{enumerate}
  \item Any action of $G$ on $F$ in the homotopy category of
    $E_\infty$-ring spectra has a unique lift, up to weak equivalence,
    to an action of $G$ on $F$ in the category of $E_\infty$-ring spectra.
  \item If $F$ and $F'$ are equipped with actions of $G$ in the
    category of $E_\infty$-ring spectra, then the natural map
\[
\pi_0 \Map^G_{E_\infty}(F,F') \to
\left(Hom_{\psth\alg}(K^\vee_0 F,K^\vee_0 F')\right)^G,
\]
from the connected components of the derived mapping space of
$G$-equivariant $E_\infty$-maps to $G$-equivariant maps of
$\psth$-algebras, is an isomorphism.
  \end{enumerate}
\end{cor}

\begin{proof}
This combines Lemma~\ref{lem:groupactions}
and the previous proposition.
\end{proof}

We now state the main result allowing construction of $K(1)$-local
objects and maps realizing algebraic data.

To help the reader navigate through the following statement, we remark
that in the application $F$ will be the $K(1)$-localization of some
$K(2)$-local $E_\infty$-ring spectrum and $T$ will be constructed as
be a $K(1)$-local object together with a chromatic attaching map $T\to
F$.

\begin{thm}
\label{thm:realizability}
Suppose $F$ is a $K(1)$-local, complex oriented $E_\infty$-ring
spectrum such that $F_*$ is $p$-torsion free and concentrated in
degrees congruent to $0$ mod $(2p-2)$.  Let $T_*$ be a $p$-adic subring of
$F_*$, with $T_0$ the $p$-adic completion of a smooth $\mb
Z_p$-algebra.  Suppose further that $T_0 \subseteq F_0$ is closed
under the operation $\theta$, that the orientation  $MU_* \to F_*$
factors through $T_* \subseteq F_*$, and that the induced formal
group law on $T_*/(p)$ is of strict height $1$.

Then there exists a $K(1)$-local, complex oriented
$E_\infty$-ring spectrum $T$, equipped with a map $T \to F$ of
$E_\infty$-ring spectra, such that the induced map on homotopy
groups is the given inclusion $T_* \to F_*$ and such that the
orientation $MU \to F$ factors through an orientation $MU \to
T$.

As a $K(1)$-local, complex orientable $E_\infty$-ring spectrum, $T$ is
uniquely determined up to equivalence by $T_*$, the formal group over
$T_*$ specified by the given map $MU_*\to T_*$, and the ring
endomorphism of $T_0$ given by $x\mapsto x^p+p\theta(x)$.  The map of
$E_\infty$-ring spectra $T\to F$ is determined up to homotopy by its
effect on homotopy groups mentioned above.
\end{thm}

\begin{proof}
  We will first construct $T(p)_*$ to be the first stage in the Igusa
  tower for $T_*$. Denote by $\mb G$ the formal group over $T_*$
  classified by the given map $MU_*\to T_*$.  Now, we define $T(p)_*$
  to be the $T_*$-algebra which is the coordinate ring ${\cal O}_{\mb
    G^\vee[p] \setminus 0}$ parameterizing isomorphisms between the
  $p$-torsion $\mb G[p]$ of the formal group law on $T_*$ and the
  group of $p$-th roots of unity (see Lemma~\ref{lem:igusa},
  item~\ref{item:4}).  This is a Galois extension of $T_*$ 
  with Galois group $(\mb Z/p)^\times$.  Naturality of the Igusa tower
  implies that the map $T_* \to F_*$ induces an equivariant
  isomorphism $F(p)_* \cong T(p)_* \otimes_{T_*} F_*$.  The assumption
  that $F_*$ is concentrated in certain degrees implies that $F(p)_0 =
  F_0$ and $T(p)_0 = T_0$.

We define graded rings 
\begin{align*}
A_* &= \comp{(K_* \otimes_{MUP_*} MUP_* MUP \otimes_{MUP_*} T(p)_*)}_p\text{, and}\\
B_* &= \comp{(K_* \otimes_{MUP_*} MUP_* MUP \otimes_{MUP_*} F(p)_*)}_p.
\end{align*}
Both rings $A_*$ and $B_*$ are $p$-complete and even-periodic; this
follows from the assumption on $F_*$.  Both $A_*$ and $B_*$ have
actions of $(\mb Z/p)^\times$ through the rightmost factors in the
tensor product.  In addition, they carry actions of $\mb Z_p^\times$
through Adams operations given by the coaction of the $p$-completed
Hopf algebroid $\comp{(K_* \otimes_{MUP_*} MUP_*MUP \otimes_{MUP_*}
  K_*)}_p$.  Specifically, for any $p$-adic unit $k\in\mb Z_p^\times$,
the multiplication-by-$k$ map is a natural automorphism of any formal
group law over a $p$-adic ring, and is classified on the universal
example by a map
\[
[k]\co \comp{(MUP_* MUP)_p} \to \comp{(MUP_*)}_p.
\]
Composing this map with the coalgebra action gives compatible
  actions of $\mb Z_p^\times$ on $A_*$ and $B_*$.

By Lemma~\ref{lem:p-complete} there is an isomorphism, equivariant
  for $(\mb Z/p)^\times$ and the action of the Adams operations,
between $B_*$ and $K^\vee_* F(p)$.

As $T_0$ is closed under the operation $\theta$, the subring $A_0
\subseteq B_0$ is also closed under the operation $\theta$, as
follows.  The ring $A_0$ is universal among $p$-adic
$T(p)_0$-algebras $j\co T(p)_0 \to R$ equipped with an isomorphism
$\widehat{\mb G}_m \to j^* \mb G$.  Over the ring $B_0$, we have a
composite isogeny
\[
\widehat{\mb G}_m \overto^{\cong} \mb G \to \psi^* \mb G 
\] 
of degree $p$ by Theorem~\ref{thm:level-structures}.  As the
formal multiplicative group has only one subgroup scheme of rank $p$,
this factors uniquely as
\[
\widehat{\mb G}_m \overto^{[p]} \widehat{\mb G}_m\overto^{\cong} \psi^* \mb G,
\]
and the universal property then shows that the map
$\psi\co B_0 \to B_0$ restricts to a map $A_0 \to A_0$.  The universal
property also implies that this is a lift of Frobenius, and hence,
since $A_0$ is $p$-torsion free, it has the structure of a
$\psth$-algebra.  The action of $(\mb Z/p)^\times$ commutes with
this structure.

Lemma~\ref{lem:indgalois}, item~\ref{item:canonicalgalstructure}
implies that the maps $T(p)_0 \to A_0$ and $F(p)_0 \to B_0$ are
ind-Galois extensions of $p$-adic rings with Galois group $\mb
Z_p^\times$.  We conclude that $A_*$ is the $p$-adic
completion of an ind-\'etale extension of $T(p)_*$, that $T(p)_0$
is the $p$-adic completion of a smooth $\mb Z_p$-algebra, and that the
continuous cohomology of $\mb Z_p^\times$ with coefficients in $A_*$
vanishes.  Applying Corollary~\ref{cor:psth-realization} to $A_*$
implies that there exists a $K(1)$-local $E_\infty$-ring spectrum
$T(p)$ equipped with an isomorphism $K^\vee_* T(p) \cong A_*$ as
$\psth$-algebras.

The homotopy of $T(p)$ is $T(p)_*$ by Theorem~\ref{thm:kn-anss},
  and is concentrated in even degrees.  Therefore, $T(p)$ and $F(p)$
  are even-periodic and complex orientable, their homotopy groups are
  $p$-torsion free, and $T(p)_0$ is the $p$-adic completion of a
  smooth $\mb Z_p$-algebra.

Proposition~\ref{prop:psth-mapping} implies that the inclusion $A_*
\to B_*$ lifts to a unique map of $E_\infty$-ring spectra $f\co T(p)
\to F(p)$.  As the inclusion is $(\mb Z/p)^\times$-equivariant,
Corollary~\ref{cor:psth-action} implies that there exists, up to weak
equivalence, a unique lift of this action of $(\mb Z/p)^\times$
to $T(p)$, and a lift of $f$ to an equivariant map, where the
action on the range is the action on $F(p)$ coming from the Igusa
tower.

The description of the $K$-theory of $T(p)$, together with its Adams
operations, determines the formal group of $T(p)$ as that coming from
the given map $MU_* \to T(p)_*$.  In fact, for any $K(1)$-local
complex oriented multiplicative cohomology theory $F$, the action of
the Adams operations on $K^\vee_* F$ precisely provides descent data
for the formal group of multiplicative type carried by $F_*$.
(One can instead view this in terms of the dual, \'etale group, which
is locally isomorphic to $\mb Q_p/\mb Z_p$ and is classified by a
character of the fundamental group with coefficients in $\mb
Z_p^\times$.) Therefore, since this formal group law is lifted from
one over $T_* \subseteq T(p)_*$, by
Corollary~\ref{cor:orientation-descent} the orientation $MU \to F(p)$
lifts to an orientation $MU \to T(p)$.

Let $T$ be the homotopy-fixed point spectrum $T(p)^{h(\mb
  Z/p)^\times}$.  We have a map $T \to F$ of $E_\infty$-ring spectra,
and on homotopy groups this realizes the map of fixed subrings
$T(p)_*^{(\mb Z/p)^\times} \to F(p)_*^{(\mb Z/p)^\times}$.  By
Lemma~\ref{lem:igusa}, item~\ref{item:3} this is precisely the map
$T_* \to F_*$.

We then find that we have a composite splitting
\[
\Map_{E_\infty}(T,F) \to \Map_{E_\infty}^{(\mb Z/p)^\times}(T(p), F(p))
\to \Map_{E_\infty}(T,F)
\]
given by first taking the first stage in the Igusa tower, followed by
taking homotopy fixed points.  Upon applying $\pi_0$, homotopical
uniqueness of the equivariant lift $T(p) \to F(p)$ then implies
uniqueness of the map $T \to F$.

Corollary~\ref{cor:fgl-lifting} implies that the orientation $MU \to
T(p)$ has a unique lift $MU \to T$.
 \end{proof}

\section{$K(1)$-local realization}
\label{sec:k1-localization}

In this section, we will begin with a realization problem $(A,\mb G)$
as in Definition~\ref{defn:realization}, together with the unique
solution $R_{K(2)}$ to the realization problem for $A_{K(2)}$ as in
Theorem~\ref{thm:k2local-exists}.  We will extend this to study
solutions to the realization problem for $\comp{A}_p$, using the
$K(1)$-local $E_\infty$-ring spectrum
$R_{K(2),K(1)}=L_{K(1)}R_{K(2)}$.

\subsection{$K(1)$-localized algebras}
The following result will verify the assumptions of
Theorem~\ref{thm:realizability} and will lead to the construction of
the chromatic attaching map $R_{K(1)}\to R_{K(2),K(1)}$ in Section
\ref{sec:application-bpp2}.

\begin{prop}
\label{prop:stability}
Suppose that $(A, \mb G)$ is a realization problem, and $R_{K(2)}$
is the solution to the realization problem for $A_{K(2)}$ given by
Theorem~\ref{thm:k2local-exists}.  Let $F$ be the $K(1)$-localization
$R_{K(2),K(1)} = L_{K(1)} R_{K(2)}$, and let $B = A_{K(1)}$.
For brevity, we define $x$ to be $v_1^{p+1} v_2^{-1}$.
\begin{itemize}
\item The spectrum $F$ is a $K(1)$-local complex oriented
  $E_\infty$-ring spectrum, with $F_*$ concentrated in even degrees
  and $p$-torsion free.
\item The ring $B$ maps identically into a $p$-adic subring of $F_*$, 
with $B_0 \cong \comp{\mb Z[x^{-1}]}_p$ the $p$-adic
completion of a smooth $\mb Z_p$-algebra.
\item The orientation $MU_*\to F_*$ factors through
  $B\subseteq F_*$.
\end{itemize}
\end{prop}

\begin{proof}
We consider $F$ complex oriented by the composite
\begin{equation}\label{eq:R12-orient}
MU\to R_{K(2)}\to R_{K(2),K(1)}=F,
\end{equation}
where the first map is the orientation which is part of the 
solution to our realization problem for $A_{K(2)}$ and
the second map is the canonical localization map.

We first compute the homotopy of the $K(1)$-localization $F$.
  Application of \cite[Corollary 1.5.5]{hovey-vnelements}
allows us to conclude that the $K(1)$-localization is
$F=\comp{(v_1^{-1} R_{K(2)})}_p$.  We then have the following
(see Remarks~\ref{rmk:ringcalculations} and
\ref{rem:homotopyoflocalizations}):
\[
F_* \cong (A_{K(2)})_{K(1)} \cong \left(\comp{\mb Z_p\laur
  {x}}_p\right)[v_1^{\pm 1}]
\]
From this we find that the map
\[
B=A_{K(1)} = \comp{\mb Z_p[x^{-1}]}_p[v_1^{\pm 1}] \to F_*
\]
is injective, and the orientation factors as $MU_* \to A
\hookrightarrow B \hookrightarrow F_*$, where the first map is
part of our realization problem.
\end{proof}

\subsection{Application of obstruction theory}
\label{sec:application-bpp2}

\begin{thm}
\label{thm:k1local-exists}
Suppose that $(A, \mb G)$ is a realization problem, and that
$R_{K(2)}$ is the solution to the lifting problem for $A_{K(2)}$ given
by Theorem~\ref{thm:k2local-exists}.  Let $F$ be the
$K(1)$-localization $R_{K(2),K(1)}$, and let $B \subseteq F_*$ be the
subring which is the image of $A_{K(1)}$ as in Proposition
\ref{prop:stability}.

Then $B_0$ is closed under the power operation $\theta$, defined on
$F_0$ as in section \ref{subsect:obstruction}, if and only if there
exists a $K(1)$-local $E_\infty$-ring spectrum $T$ solving the
realization problem for $B$ together with a ``chromatic attaching
map'' of $E_\infty$-ring spectra $\alpha^{chrom}\co T\to F$, such that
the induced map on homotopy groups is isomorphic to the inclusion $B
\to F_*$.  Any two such spectra $T$ are equivalent as $E_\infty$-ring
spectra over $F$.

Under these circumstances, the orientation $MU \to F$ factors as
$MU\to T \to F$.
\end{thm}

\begin{proof}
The necessity follows from the naturality of the operation
$\theta$.  To show sufficiency, Proposition~\ref{prop:stability},
together with the assumptions on the power operation $\theta$, show
that the subring $B \subseteq F_*$ satisfies the assumptions of
Theorem~\ref{thm:realizability}.
\end{proof}

\begin{defn}
\label{def:k12-def}
Suppose that $(A, \mb G)$ is a realization problem such that the
image of $\left(A_{K(1)}\right)_0$ in $\left(A_{K(2),K(1)}\right)_0$ is closed under the
power operation $\theta$.  Then let $R_{K(1) \vee K(2)}$ be the
$E_\infty$-ring spectrum formed as the homotopy pullback in the
following diagram:
\begin{equation}\label{diagram}
\xymatrix{
R_{K(1) \vee K(2)} \ar@{.>}[r] \ar@{.>}[d] & R_{K(2)} \ar[d]^\iota \\
R_{K(1)} \ar[r]^(.4){\alpha^{chrom}} & R_{K(2),K(1)}. }
\end{equation}
Here, $R_{K(2)}$ is as in Theorem~\ref{thm:k2local-exists}, $\iota$ is
the canonical localization map, and $\alpha^{chrom}$ is the
chromatic attaching map from Theorem~\ref{thm:k1local-exists}. 
\end{defn}

\begin{prop}
\label{prop:12homotopy}
Under the conditions of Definition~\ref{def:k12-def}, the connective
cover of $R_{K(1) \vee K(2)}$ solves the realization problem for
$\comp{A}_p \cong \mb Z_p[v_1,v_2]$, and any two such solutions
are equivalent.
\end{prop}

\begin{proof}
The forgetful functor from $E_\infty$-ring spectra to spectra
preserves homotopy limits.  The homotopy groups of all spectra
involved in the solid part of (\ref{diagram}) are concentrated
in even degrees, and $\pi_*(\iota)$ and $\pi_*(\alpha^{chrom})$
are injective.  Therefore, for all $n\in\mb Z$ we have
\[ \pi_{2n}R_{K(1)\vee K(2)}=\pi_{2n}R_{K(2)}\cap \pi_{2n}R_{K(1)}\mbox{ and }\]
\[ \pi_{2n-1}R_{K(1)\vee K(2)}=\coker(\pi_{2n}R_{K(2)}\oplus\pi_{2n}R_{K(1)}\to\pi_{2n}R_{K(2),K(1)}).\]

Explicitly, we have the following.
\begin{align*}
\pi_{2n} R_{K(2)} &= \left\{\sum_{k+(p+1)l = n,\ k \geq 0}c_{k,l} v_1^k
    v_2^l\ \middle |\ c_{k,l} \in \mb Z_p\right\}\\
&=  \left\{
    \sum_{0\leq k\equiv n\, (p+1)} c_k v_1^kv_2^{\frac{n-k}{p+1}}\,\middle|\, c_k\in\mb Z_p\right\}
\end{align*}
\begin{align*}
\pi_{2n} R_{K(1)} &= \left\{\sum_{k+(p+1)l = n,\ l \geq 0}c_{k,l} v_1^k
    v_2^l\ \middle |\ c_{k,l} \in \mb Z_p,\ c_{k,l} \rightarrow 0\text{ as } k\rightarrow -\infty\right\}\\
&=\left\{ \sum_{0\leq l}c_{l}v_1^{n-(p+1)l}v_2^l\,\middle|\, c_l\in\mb Z_p, \, c_l\to 0\right\}
\end{align*}
\begin{align*}
\pi_{2n} R_{K(2),K(1)} &= \left\{\sum_{k+(p+1)l = n}c_{k,l} v_1^k
v_2^l\ \middle |\ c_{k,l} \in \mb Z_p,\ c_{k,l} \rightarrow 0\text{ as }
k\rightarrow -\infty\right\}\\
&=\left\{ \sum_{k\in n+(p+1)\mb Z}
  c_kv_1^kv_2^{\frac{n-k}{p+1}}\,\middle|\, c_k\in\mb Z_p, \, c_k\to 0\mbox{ for } k\to -\infty\right\}.
\end{align*}

If $n \geq 0$, every element in $\pi_{2n} R_{K(2), K(1)}$ is a sum of
elements in the images of the other two groups. 
Therefore, for all $n\ge 0$ we have $\pi_{2n-1}R_{K(1)\vee K(2)}=0$
and
\[
\pi_{2n} R_{K(1) \vee K(2)} = \left\{\sum_{k+(p+1)l = n,\ k \geq 0,\ l \geq 0}c_{k,l} v_1^k
    v_2^l\,\middle|\,c_{k,l} \in \mb Z_p\right\},
\]
which is the portion of $\comp{A}_p \cong \mb Z_p[v_1,v_2]$
concentrated in degree $2n$.  Hence the connective cover $R_{K(1)\vee
K(2)}[0,\infty)$, which we denote by $\comp{R}_p$, is an
$E_\infty$-ring spectrum with homotopy $\pi_*\comp{R}_p\simeq
\comp{A}_p$.  Since this is concentrated in even dimensions,
$\comp{R}_p$ is complex orientable and it remains to see that the
resulting formal group is the one given in the realization problem.

As $MU$ is connective, orientations of $\comp{R}_p$ are in
bijective correspondence with orientations of $R_{K(1) \vee K(2)}$.
Moreover, we know that the formal group laws of $R_{K(1)}$ and
$R_{K(2)}$ are the formal group laws pushed forward from $A$.
Applying Lemma~\ref{lem:fgl-pullback}, we find that $\comp{R}_p$ is
complex orientable with formal group law pushed forward from the map
$A \to \pi_* \comp{R}_p$.

This shows that $\comp{R}_p$ is a solution to the realization
problem for $\comp{A}_p$, and we now address uniqueness.

Given any other solution $R'$, Theorem~\ref{thm:k2local-exists}
shows that there exists an equivalence of realizations $R_{K(2)} \to
R'_{K(2)}$, and Theorem~\ref{thm:k1local-exists} shows that there
exists an equivalence $R_{K(1)} \to R'_{K(1)}$ of realizations
compatible with the chromatic attaching maps.  Taking the connective
cover of the induced equivalence $R_{K(1) \vee K(2)} \to R'_{K(1) \vee
K(2)}$ gives the desired result.
\end{proof}

\section{$K(0)$-local realization}
\label{sec:k0-localization}

\begin{thm}
\label{thm:newmain}
Suppose that $(A, \mb G)$ is a realization problem.
Then $(A, \mb G)$ has a solution if and only if the degree-zero
portion $(A_{K(1)})_0\subseteq (A_{K(2),K(1)})_0$ is stable under the
power operation $\theta$. If such a solution exists, it is unique
up to equivalence.
\end{thm}

\begin{proof}
It is clear that closure under $\theta$ is a necessary condition, so let 
us assume it holds and establish that our realization problem has
a solution, unique up to equivalence.
Let $R_{\mb Z_p}=\comp{R}_p$ be the solution to the realization
problem for $\comp{A}_p$ from Proposition~\ref{prop:12homotopy}, and
define $R_{\mb Q_p}$ to be the rationalization of $R_{\mb Z_p}$.  Let 
$R_{\mb Q}$ be the free $E_\infty$-$\eilm{\mb Q}$-algebra on 
$S^{2p-2}\vee S^{2p^2-2}$. Then $\pi_*R_{\mb Q}$ is isomorphic
to $A_{K(0)} \cong \mb Q[v_1,v_2]$, and $R_{\mb Q}$ is a solution
to the realization problem for $A_{K(0)}$.

Proposition~\ref{prop:12homotopy} and rational homotopy theory imply
that there exists an arithmetic attaching map of $E_\infty$-$\eilm{\mb
  Q}$-algebras
\[
\alpha^{arith}: R_{\mb Q} \to R_{\mb Q_p}
\]
inducing the inclusion $\mb Q[v_1,v_2] \to \mb Q_p[v_1,v_2]$ on
homotopy groups.

Let $R$ be the $E_\infty$-ring spectrum
formed as the homotopy pullback in the following diagram:
\[
\xymatrix{
R \ar@{.>}[r] \ar@{.>}[d] & R_{\mb Z_p} \ar[d] \\
R_{\mb Q} \ar[r]^(.4){\alpha^{arith}} & R_{\mb Q_p}
}
\]
In a similar manner to the proof of Proposition~\ref{prop:12homotopy},
we find that the pullback diagram defining $R$ gives rise to the
following Mayer-Vietoris square on homotopy groups:
\[
\xymatrix{
  R_* \ar[r] \ar[d] &
  \mb Z_p[v_1,v_2] \ar@{^{(}->}[d]\\
  \mb Q[v_1,v_2] \ar@{^{(}->}[r] &
  \mb Q_p[v_1,v_2]
}
\]
Observe that, since $\mb Q_p=\mb Q+\mb
Z_p$, we have $\pi_{2n-1}R=0$ for all $n\ge 0$.  The above diagram is
then a pullback diagram, and the homotopy groups of
$R$ map isomorphically to the subring which is the image of $A = \mb
Z_{(p)}[v_1,v_2]$.
The formal group laws of $R_{\mb Z_p}$ and $R_{\mb Q}$ are the
formal group laws pushed forward from $A$.  Applying
Lemma~\ref{lem:fgl-pullback}, we find that $R$ is complex orientable
with formal group law pushed forward from the isomorphism $A \to
R_*$.  Therefore, $R$ is a solution to the realization problem for
$A$ itself.

Given any other solution $T$ to the realization problem,
Proposition~\ref{prop:12homotopy} implies that there exists an
equivalence $R_{\mb Z_p} \to \comp{T}_p$ of realizations for
$\comp{A}_p$.  As $R_{\mb Q}$ is free, in order to construct a
map of $E_\infty$-$\mb{HQ}$-algebras $R_{\mb Q} \to T_{\mb Q}$
compatible with the arithmetic attaching maps it suffices for the
images of the generators $v_1, v_2$ in $\pi_*T_{\mb Q_p}$
to lift to $\pi_* T_{\mb Q}$.  However, this follows
because an equivalence of realizations preserves the subring of
$\comp{A}_p$ which is the image of $A$.
\end{proof}

\section{An algebraic description of $\theta$ and the proof of Theorem \ref{thm:main}}\label{sec:proofofmain}

In Section~\ref{subsec:algebraictheta} we will give an
algebraic characterization of the power operation $\theta$ figuring in
Theorem \ref{thm:newmain}.  Section~\ref{subsec:proofofmainstability}
uses this, and the moduli of generalized elliptic curves with
$\Gamma_1(3)$-level structure, to produce a realization problem
at the prime 2 meeting the assumptions of Theorem
\ref{thm:newmain}.  This will complete the proof of
Theorem~\ref{thm:main}.

\subsection{Computing $\theta$}\label{subsec:algebraictheta}
Let $(A,\mb G)$ be a (generalized $\BPP{2}$-)realization problem.
Crucially using the Goerss-Hopkins-Miller theorem, we obtained an
operation $\theta$ on $(A_{K(2),K(1)})_0$, which is a
mixed-characteristic discrete valuation ring.  Since this ring is
torsion-free, $\theta$ is determined by the ring endomorphism
\[
x \mapsto (\psi^p)^{top}(x) := x^p+p\cdot\theta(x)
\]
of $(A_{K(2),K(1)})_0$.  This map $(\psi^p)^{top}$ is a lift
of Frobenius: a ring endomorphism with $(\psi^p)^{top}(x)\equiv x^p$
mod $p$.  In this section we will explain how to define
$(\psi^p)^{top}$ without homotopy theory, and in particular
without solving the realization problem for $A_{K(2)}$.  This is
essentially a consequence of work of Ando-Hopkins-Strickland relating
the power operations on Lubin-Tate spectra to their naturally defined
descent data for level structures \cite[Section
12.4]{ando-hopkins-strickland}.

Recall the associated Lubin-Tate ring $B \supseteq A_{K(2)}$ from
  Definition~\ref{def:associated-lubintate}, a $G$-Galois
extension for the group $G=\mb
F_{p^2}^\times\rtimes\mathrm{Gal}(\mb F_{p^2}/\mb
F_p)$.  Form the pushout of rings
\[
\xymatrix{ A_{K(2)}\ar@{^(->}[rr]^G \ar[dd] & & B\ar[dd]\\
 & & \\ A_{K(2),K(1)} \ar@{^(->}[rr]^G & & \tilde{B}}
\]
along the canonical ring homomorphism $A_{K(2)}\to A_{K(2),K(1)}$. The
base change $\mb G_B$ over $B$ of the given formal group $\mb G$ is a
Lubin-Tate formal group by Remark~\ref{rmk:lubintate}, and from
Remark~\ref{rmk:ringcalculations} we see that $\tilde{B}\cong
\comp{(v_1^{-1} B)}_{p}$.  This shows that $\mb G_B$ is a
$p$-divisible group of height $2$ which is a universal deformation of
its special fiber.  The base change $\mb G_{\tilde{B}}$ of the 
$p$-divisible group $\mb G_B$ to $\tilde{B}$ is ordinary
\cite[Lemma II.1.1, 4]{harristaylor}, so there is
an extension of $p$-divisible groups over $\tilde{B}$
\[
0 \to\mb G^{for}_{\tilde{B}}\stackrel{\iota}{\to} \mb G_{\tilde{B}}\to \mb G^{et}_{\tilde{B}}\to 0
\]
with $\mb G^{for}_{\tilde{B}}$ (resp. $\mb G^{et}_{\tilde{B}}$) a
formal (resp. \'etale) $p$-divisible group of height $1$.  The kernel
$\mb G^{for}_{\tilde{B}}[p] \subseteq \mb G^{for}_{\tilde{B}}$ of
multiplication by $p$ is the unique subgroup of $\mb
G^{for}_{\tilde{B}}$ of order $p$.  We obtain a subgroup $C$ of
order $p$ as the image $\iota(\mb G^{for}_{\tilde{B}}[p])\subseteq
\mb G_{\tilde{B}}$, called the canonical subgroup. By
construction, the mod-$p$ reduction of the quotient of $\mb
G_{\tilde{B}}$ by its canonical subgroup has canonical isomorphisms
\[
\overline{\mb G_{\tilde{B}}/C}\cong \overline{\mb
  G_{\tilde{B}}}^{(p)}\cong \overline{\mb G_{\tilde{B}}}.
\]
(As in Section~\ref{subsec:heightone}, $\overline{\mb
G_{\tilde{B}}}^{(p)}$ denotes the pullback along the Frobenius
isogeny.)  The latter isomorphism occurs because $\overline{\mb
G_{\tilde{B}}}$ is defined over $\mb F_p$.  We conclude from
\cite[Proposition 7.1]{eike} that there is a unique lift of Frobenius
$\psi\co \tilde{B}\to\tilde{B}$ such that there is an isomorphism of
$p$-divisible groups over $\tilde{B}$
\begin{equation}\label{eq:characterizingpsi}
\psi^*\mb G_{\tilde{B}}\cong \mb G_{\tilde{B}}/C.
\end{equation}
Since $\mb G_{\tilde{B}}$ comes by base change from $A_{K(2),K(1)}
\subseteq\tilde{B}$, $\psi$ restricts to a lift of Frobenius
$(\psi^p)^{alg}$ on $A_{K(2),K(1)}$.  The construction of
$(\psi^p)^{alg}$ involves no homotopy theory, so the following result
achieves the goal of this section.

\begin{prop}\label{prop:algcompoftheta}
In the above situation we have $(\psi^p)^{top}=(\psi^p)^{alg}$.
\end{prop}

\begin{proof}
It is enough to see that $\psi$ is equal to
$(\psi^p)^{top}\otimes\mathrm{id}_{\tilde{B}}$. This is true
because, by equation (\ref{eq:toplevel}) in
Section~\ref{subsec:heightone}, the topologically constructed lift
of Frobenius $(\psi^p)^{top}\otimes\mathrm{id}_{\tilde{B}}$
satisfies property (\ref{eq:characterizingpsi}), which characterizes
the algebraically constructed lift $\psi$.
\end{proof}

\begin{rmk}
It is indispensable to use non-formal $p$-divisible groups to obtain
an algebraic characterization of $(\psi^p)^{top}$ as above.  The
base-change of the {\em formal} group $\mb G_B$ to $\tilde{B}$ is
$\mb G_{\tilde{B}}^{for}$, and the first results in Section
\ref{subsec:heightone} imply that one cannot recover
$(\psi^p)^{alg}$ from knowledge of $\mb G_{\tilde{B}}^{for}\subseteq
\mb G_{\tilde{B}}$ alone.
\end{rmk}

\subsection{Proof of Theorem \ref{thm:main}}\label{subsec:proofofmainstability}

In this section we complete the proof of Theorem \ref{thm:main}
by constructing a realization problem $(A,\mb G)$
at the prime $2$ which meets the assumptions of Theorem
\ref{thm:newmain}.

Define $A$ to be the graded ring $\mb Z_{(2)}[a,b]$ with grading
$|a|=2,|b|=6$.  Consider the following Weierstrass equation
over $A$:
\begin{equation}\label{eq:gamma13}
y^2+a\cdot xy+b\cdot y= x^3.
\end{equation}

We remark this is a universal generalized elliptic curve with a
$\Gamma_1(3)$-level structure; this connection will be 
exploited in \cite{tmforientation}.

Equation (\ref{eq:gamma13}) determines a graded formal group law
$\mb G$ over $A$ with respect to the coordinate $-x/y$.

\begin{prop}
\label{prop:realproblem}
This pair $(A,\mb G)$ is a realization problem at the prime $2$.
\end{prop}

\begin{proof}
From \cite[Lemma 1]{laures-k1-local-tmf}, the
$2$-primary Hazewinkel generators $v_1$ and $v_2$ are
mapped to $a$ and $b$ respectively.
\end{proof}

Resuming notation from Section~\ref{subsec:algebraictheta},
our remaining task is to compute $\theta$ (or equivalently
$(\psi^2)^{alg}$) for the present realization problem.  We observe
that by defining $t$ to be $au^{-1}$, where $a$ is the image of
$a\in A$ and $u\in B_2$ satisfies $u^3 = b$, we have
\[
B_0=\mb W(\mb F_4)\pow{t}.
\]
Over $B_0$, the Weierstrass curve of
Equation~(\ref{eq:gamma13}) becomes isomorphic to the
following curve:
\[
E: y^2+t\cdot xy+y=x^3
\]
The key computational result, which at present lacks an analogue
at primes other than $2$, is the following.

\begin{prop}[{\cite[Section 3]{rezk-height2}}]
\label{prop:theta}
The universal isogeny of degree $2$ with domain $E$ is defined
over $B_0[d]/(d^3 - td - 2)$, and has range the elliptic curve
\[
y^2 + (t^2 + 3d - td^2) xy + y = x^3.
\]
The kernel of this isogeny is generated by the $2$-torsion point with
coordinates $(x/y, 1/y) = (-d, d^3)$.
\end{prop}

This first implies that there is a unique $\alpha\in
\tilde{B}=\comp{(t^{-1}B_0)}_{2}$ satisfying
$\alpha^3-t\alpha-2=0$, because solutions of this cubic equation
define points of order $2$ on the curve.  There is a unique such point
$P\in E(\tilde{B})$ defined over the complete discrete valuation
ring $\tilde{B}$ because the curve is ordinary there.  (The
existence of a unique $\alpha$ can also be checked directly using
Hensel's lemma.)  Second, since $\left((\psi^2)^{alg}\right)^* E\cong
E/\{ 1,P\}$, Rezk's computations imply that we have
\[
(\psi^2)^{alg}(t)=t^2+3\alpha-\alpha^2t = t^{-1}(t^3 + 3\alpha t -
  (\alpha t)^2).
\]

The element $\alpha \in\tilde{B}$ is divisible by two.  Defining $c$
to be $\frac{1}{2}\alpha t\in\tilde{B}$, the above cubic equation
$\alpha^3-t\alpha-2=0$ for $\alpha$ is equivalent to the equation
\[
c=-1+4t^{-3}c^3.
\]
This can recursively be solved to yield 
\[
c=-1-4t^{-3}-48\cdot t^{-6}-192\cdot t^{-9}\cdots\in\comp{\mb Z_2[t^{-3}]}_2.
\]
With our present choice of coordinate, the inclusion
$(A_{K(1)})_0 \subseteq (A_{K(2),K(1)})_0$
is the inclusion of the subring $\comp{\mb
Z_2[t^{-3}]}_2\subseteq \comp{\mb Z_2\laur{t^3}}_2$.  Combining the above,
one verifies that 
$(\psi^2)^{alg}(t^{-3})\in \comp{\mb Z_2[t^{-3}]}_2$. In
  combination with Proposition \ref{prop:algcompoftheta}
this establishes the following.
\begin{prop}
The subring $(A_{K(1)})_0\subseteq (A_{K(2),K(1)})_0$ is
  stable under $\theta$. 
\end{prop}
In combination with Theorem~\ref{thm:newmain}, this proves Theorem
\ref{thm:main}.

%\nocite{*}
\bibliography{level3}

\end{document}